\documentclass[a4paper,10pt,leqno,twoside]{tobart}

\usepackage[english]{babel} \usepackage{inputenc, amsmath, amssymb, latexsym,
  epic, epsfig, rotating, fancyhdr, amsthm, pifont, enumitem, graphicx, subfigure}

\newcommand{\ld}{\ensuremath{,\ldots,}}
\newcommand{\ssq}{\ensuremath{\subseteq}}
\newcommand{\smin}{\ensuremath{\setminus}}
\newcommand{\eps}{\ensuremath{\varepsilon}}

\newcommand{\N}{\ensuremath{\mathbb{N}}}
\newcommand{\R}{\ensuremath{\mathbb{R}}}

\newcommand{\C}{\ensuremath{\mathbb{C}}}

\newcommand{\D}{\ensuremath{\mathbb{D}}}

\newcommand{\graph}{\ensuremath{\mathrm{graph}}}

\newcommand{\kreis}{\ensuremath{\mathbb{T}^{1}}}

\newcommand{\sltr}{\ensuremath{\textrm{SL}(2,\mathbb{R})}}

\newcommand{\twomatrix}[4]{\ensuremath{\left(\begin{array}{cc} #1 & #2 \\ #3 &
      #4 \end{array}\right)}}

\newcommand{\alphlist}{\begin{list}{(\alph{enumi})}{\usecounter{enumi}\setlength{\parsep}{2pt}
      \setlength{\itemsep}{1pt} \setlength{\topsep}{5pt}
      \setlength{\partopsep}{3pt}}}

\newcommand{\arablist}{\begin{list}{(\arabic{enumi})}{\usecounter{enumi}\setlength{\parsep}{2pt}
          \setlength{\itemsep}{1pt} \setlength{\topsep}{5pt}
          \setlength{\partopsep}{3pt}}}

\newcommand{\romanlist}{\begin{list}{(\roman{enumi})}{\usecounter{enumi}\setlength{\parsep}{2pt}
              \setlength{\itemsep}{1pt} \setlength{\topsep}{5pt}
              \setlength{\partopsep}{3pt}}}

 \newcommand{\listend}{\end{list}}

\newcommand{\bulletlist}{\begin{list}{$\bullet$}{\setlength{\parsep}{2pt}
                \setlength{\itemsep}{1pt} \setlength{\topsep}{5pt}
                \setlength{\partopsep}{3pt}\setlength{\leftmargin}{15pt}}}

\newcommand{\foot}{\footnote}

\newcommand{\nfolge}[1]{\ensuremath{(#1)_{n\in\mathbb{N}}}}

\newcommand{\ncap}{\ensuremath{\bigcap_{n\in\N}}}

\newcommand{\nLim}{\ensuremath{\lim_{n\rightarrow\infty}}}

\newcommand{\kLim}{\ensuremath{\lim_{k\rightarrow\infty}}}

\newcommand{\nKonv}{\ensuremath{\stackrel{n\rightarrow
      \infty}{\longrightarrow}}}

\newcommand{\insum}{\ensuremath{\sum_{i=1}^n}}

\newcommand{\inergsum}{\ensuremath{\sum_{i=0}^{n-1}}}

\newcommand{\knergsum}{\ensuremath{\sum_{k=0}^{n-1}}}

\newcommand{\ntel}{\ensuremath{\frac{1}{n}}}

\newcommand{\halb}{\ensuremath{\frac{1}{2}}}

\newtheoremstyle{tobthm}{3pt}{3pt}{\itshape}{0pt}{\bfseries}{.}{0.5eM}{}
\theoremstyle{tobthm}
\newtheorem{definition}{Definition}[section]
\newtheorem{thm}[definition]{Theorem}
\newtheorem{theorem}[definition]{Theorem}

\newtheorem{lem}[definition]{Lemma}
\newtheorem{lemma}[definition]{Lemma}
\newtheorem{cor}[definition]{Corollary}
\newtheorem{prop}[definition]{Proposition}
\newtheorem{proposition}[definition]{Proposition}

\newtheoremstyle{tobrem}{3pt}{3pt}{\normalfont}{0pt}{\bfseries}{.}{0.5em}{}
\theoremstyle{tobrem} 
\newtheorem{rem}[definition]{Remark} 

\newtheorem{questions}[definition]{Questions}

\newcommand{\base}{\ensuremath{\gamma}}

\newcommand{\twovector}[2]{\ensuremath{\left(\begin{array}{c} #1 \\ #2 \end{array}\right)}}

\numberwithin{equation}{section}
\numberwithin{figure}{section}

\title{\Large\textsc{A model for the nonautonomous Hopf bifurcation}}
\author{V.~Anagnostopoulou \and T.~J\"ager \and G.~Keller}

\newcommand{\cM}{\mathcal{M}}

\newcommand{\lmax}{\ensuremath{\lambda_{\mathrm{max}}}}

\newcommand{\Romanlist}{\begin{list}{(\Roman{enumi})}{\usecounter{enumi}\setlength{\parsep}{2pt}
              \setlength{\itemsep}{1pt} \setlength{\topsep}{5pt}
              \setlength{\partopsep}{3pt}}}

\begin{document}

\setlength{\abovedisplayskip}{0.8ex}
\setlength{\abovedisplayshortskip}{0.6ex}

\setlength{\belowdisplayskip}{0.8ex}
\setlength{\belowdisplayshortskip}{0.6ex}

\maketitle

\abstract{Inspired by an example of Grebogi {\em et al}
  \cite{grebogi/ott/pelikan/yorke:1984}, we study a class of model
  systems which exhibit the full two-step scenario for the
  nonautonomous Hopf bifurcation, as proposed by Arnold
  \cite{arnold:1998}. The specific structure of these models allows a
  rigorous and thorough analysis of the bifurcation pattern. In
  particular, we show the existence of an invariant {\em `generalised
    torus'} splitting off a previously stable central manifold after
  the second bifurcation point.

  The scenario is described in two different settings. First, we
  consider deterministically forced models, which can be treated as
  continuous skew product systems on a compact product space.
  Secondly, we treat randomly forced systems, which lead to skew
  products over a measure-preserving base transformation. In the
  random case, a semiuniform ergodic theorem for random dynamical
  systems is required, to make up for the lack of compactness.
  \medskip

  \noindent {\em 2010 Mathematics Subject Classification.} Primary
  39A28, 37H20, 34C23}
  \smallskip

 \noindent{\em Keywords: Skew Products, Random Dynamical Systems,
  Nonautonomous Hopf Bifurcation.}\\

\noindent
\section{Introduction}

External forcing often leads to important changes in the bifurcation
pattern of dynamical systems. Yet, despite the relevance of this
issue in many applications and significant progress over the last
decades
(see \cite{arnold:1998,rasmussen2007attractivity,Poetzsche2011bifurcations}
for an overview and
\cite{rasmussen2006towards,rasmussen2007nonautonomous,zmarrou/homburg:2007,zmarrou/homburg:2008}
for some recent advances), our understanding of non-autonomous
bifurcations is still limited.  Maybe the most prominent example for
this is the non-autonomous Hopf bifurcation
\cite{arnold:1998,johnson/kloeden/pavani:2002,BottsHomburgYoung2011hopf}.
Here, external forcing can lead to the separation of the
complex-conjugate eigenvalues \cite{arnoldetal:1986}.  This gives rise
to a two-step bifurcation scenario, in which an invariant {\em
  `torus'} splits off a previously stable central manifold
\cite[Chapter 9.4]{arnold:1998}. However, so far this phenomenological
description is mainly based on numerical evidence, and up to date
there exist no non-trivial examples for which this bifurcation pattern
can be described analytically. In particular, it is an open problem to
describe the structure of the split-off {\em `torus'}. Earlier
simulations suggested that this structure is simple, in the sense that
the intersection with each fibre of the product space is a topological
circle \cite{arnoldetal:1986}.\foot{In the simple case where the
  driving process is an irrational rotation on the circle, this means
  that the considered invariant set is homeomorphic to the
  two-dimensional torus. This explains our terminology.}  However,
later numerical studies based on refined algorithms indicate that more
complicated structures may appear as well
\cite{KellerOchs1999numerical} .

The aim of this article is to give a description of the non-autonomous
Hopf bifurcation in a class of model systems which is accessible to a
rigorous analysis, but at the same time allows for highly non-trivial
dynamics.  For the sake of a simpler exposition we focus on
discrete-time systems, although continuous-time analogues are easy to
derive (see Section~\ref{CtsModels}). In the situation we consider,
the split-off {\em `torus'} consists of a topological circle in each
fibre and hence belongs to the simpler case described above, but this
should not be taken as an indication for the general case.
\medskip\pagebreak

\noindent
We study parametrised families of skew products
\begin{equation}\label{f_beta}
  f_\beta \ : \ \Theta\times \R^2 \to \Theta \times \R^2\quad , \quad
  (\theta,v)\mapsto (\base(\theta),f_{\beta,\theta}(v))
\end{equation}
with {\em fibre maps}
\begin{equation} \label{f_beta-fibres} f_{\beta,\theta}(v) \ = \
  \left\{\begin{array}{cl} h(\beta\|v\|) A(\theta) \frac{v}{\|v\|} &
      \textrm{if } v\neq 0\\ \ \\0 & \textrm{if } v=0
\end{array}\right. \quad ,
\end{equation}
where $\|\cdot\|$ denotes the Euclidean norm on $\R^2$ and
$\beta\in\R^+$ is the bifurcation parameter. Maps of this type were
introduced by Grebogi {\em et al}
\cite{grebogi/ott/pelikan/yorke:1984} as examples for the existence of
strange non-chaotic attractors. A first step in their rigorous
analysis was made in \cite{glendinning/jaeger/keller:2006}, and our
results on continuous systems can be seen as an extension of this work
(see Theorem~\ref{theorem_topological} and Section~\ref{top.setting}).

We consider two different settings. For modelling {\em
  deterministic forcing}, we assume that
\begin{itemize}
\item[\bf (D1)] $\Theta$ is a compact metric space and $\base$ is a
  homeomorphism;
\item[\bf (D2)] $h:\R^+\to\R^+$ is $\mathcal{ C}^2$, strictly
  increasing, strictly concave, bounded and satisfies $h(0)=0$
  and $h'(0)=1$;
\item[\bf (D3)] $A:\Theta \to \sltr$ is continuous.
\end{itemize}

In order to give a concise description of the bifurcation pattern in
this setting, we concentrate on the behaviour of the global attractor of $f_\beta$.  By rescaling if
necessary, we may and will assume
\begin{equation} \label{e.scaling}
  \sup_{x\geq 0}h(x) \ \leq\
  \left(\max_{\theta\in\Theta}\|A(\theta)\|\right)^{-1} \ .
\end{equation}
Consequently $f_\beta(\Theta\times\R^2) \ssq \Theta\times
\overline{B_1(0)}$, so that the {\em global attractor} can be defined
as
\begin{equation}\label{e.global-attractor}
\mathcal{A}_\beta\ =\ \bigcap_{n\in\N} f_\beta^n\left(\Theta\times \overline{B_1(0)}\right) \ .
\end{equation}
We let $\mathcal{A}_\beta(\theta)=\{x\in\R^2\mid
(\theta,x)\in\mathcal{A}_\beta\}$ and use the analogous notation for
other subsets of product spaces.  As we will see, the particular
structure of \eqref{f_beta} implies that $\mathcal{A}_\beta$ has
the form
\begin{equation}
  \label{e.global-attractor-parametrisation}
  \mathcal{A}_\beta \  = \
  \left\{\left(\left.\theta,rv(\alpha)\right) \
    \right|\ \theta\in\Theta,\ \alpha\in[0,1),\ r\in[0,r_\beta(\theta,\alpha)] \right\} \ ,
\end{equation}
where $v(\alpha)=(\cos(2\pi\alpha),\sin(2\pi\alpha))^t$ and $r_\beta
:\Theta\times\R\to [0,1]$ is an upper semi-continuous function which
is \halb-periodic in the second variable.
The bifurcation parameters in the above system are determined by the
maximal exponential expansion rate of the cocycle $(\gamma,A)$. The
latter is given by the {\em maximal Lyapunov exponent} of $A$,
\begin{equation}\label{e.lmax_top}
\lmax(A) \ = \ \sup_{\theta\in\Theta} \limsup_{n\to\infty}
\frac 1n \log\|A_n(\theta)\| \ ,
\end{equation}
where $A_n(\theta) = A(\base^{n-1}\theta) \circ \ldots \circ
A(\theta)$.

\begin{theorem}\label{theorem_topological} Suppose $(f_\beta)_{\beta\in\R^+}$ is of the form
  (\ref{f_beta}) and satisfies conditions (D1)--(D3). Let
 \[ \beta_1\ :=\ e^{-\lmax} \quad \textrm{and} \quad \beta_2\ :=\ e^{\lmax} \ .
\]
Then the following hold.
\alphlist
\item If $\beta< \beta_1$, then the global attractor
  $\mathcal{A}_\beta$ is equal to $\Theta\times \{0\}$.
\item If $\beta_1<\beta<\beta_2$, then there exists at least one
  $\theta^*\in\Theta$ such that $\mathcal{A}_\beta(\theta^*)$ is a
  line segment of positive length.
\item If $\beta > \beta_2$, then for all $\theta\in\Theta$ the set
  $\mathcal{A}_\beta(\theta)$ is a closed topological disk\foot{That
    is, homeomorphic to the closed unit disk $\D=\{z\in\C\mid
    |z|\leq1\}$.} and depends continuously on $\theta$. In other words,
  the function $r_\beta$ is strictly positive and continuous.

  Further, the compact $f_\beta$-invariant set
  \begin{equation} \label{e.Theta-torus} \mathcal{T}_\beta \ = \
    \partial\mathcal{A}_\beta \ =
    \ \left\{\left(\left.\theta,r_\beta(\theta,\alpha)v(\alpha)\right) \
   \right|\ \theta\in\Theta,\ \alpha\in[0,1) \right\} \
\end{equation}
is the global attractor outside $\Theta\times\{0\}$, in the sense that
\[
\mathcal{T}_\beta \ = \ \ncap
f^n_\beta\left(\Theta\times\left(\overline{B_1(0)}\smin B_\delta(0)\right)\right)
\]
for all sufficiently small $\delta>0$.
\listend
\end{theorem}

\begin{rem}
\alphlist
\item Note that if $\lmax(A)=0$, case (b) in the theorem is void since
  then $\beta_1=\beta_2$.
\item In the intermediate region $\beta_1<\beta<\beta_2$, as well as
  for the critical cases $\beta=\beta_1$ and $\beta=\beta_2$, a great
  variety of dynamical behaviour is possible. In particular this
  behaviour is not uniform for all orbits, and given two
  $\gamma$-invariant measures $m_1$ and $m_2$ on $\Theta$ the
  typical dynamics with respect to $m_1$ and $m_2$ may be very
  different.  Therefore, the feasible approach in this parameter
  regime is to fix a $\gamma$-invariant ergodic measure $m$ on the
  base $\Theta$ and to describe the structure of
  $\mathcal{A}_\beta(\theta)$ and other relevant properties of the
  system for $m$-almost every $\theta\in\Theta$.

  However, it turns out that for such an $m$-dependent description the
  topological structure on $\Theta$ provides no additional information
  whatsoever.  Hence, all the related questions can directly be
  addressed in the purely measure-theoretic setting of random
  dynamical systems. In our context, this means that we can apply the
  random analogue to Theorem~\ref{theorem_topological}, which is given
  by Theorem~\ref{t.hopf-random} below, to obtain further information
  about the $m$-typical behaviour. See Remark~\ref{r.random-hopf}(d)
  for details. In a similar way, further information on the critical
  parameters is provided by Proposition~\ref{p.critical} below.

\item The focus on the global attractor $\mathcal{A}_\beta$ and the
  sets $\mathcal{A}_\beta(\theta)$ in the above statement corresponds
  to the concept of pullback attractors in random dynamical systems.
  It describes the behaviour of trajectories coming from $-\infty$ in
  time.  The complementary point of view is to study forward dynamics,
  meaning the asymptotic behaviour of trajectories
  $f_\beta^n(\theta,v)$ as $n$ goes to $+\infty$. In situations (a)
  and (c) of the above theorem, information about the forward dynamics
  can be derived easily. In part (a), we have
  \begin{equation}
  \nLim f_{\beta,\theta}^n(v)\ = \ 0 \quad \textrm{for all }
  (\theta,v)\in\Theta\times \R^2 \ ,
  \end{equation}
  whereas in part (c) we have
  \begin{equation}
    \nLim d\left(f^n_{\beta}(\theta,v),\mathcal{T}_\beta\right) \ = \ 0 \quad
    \textrm{for all } (\theta,v) \in \Theta\times(\R^2\smin\{0\}) \ .
  \end{equation}
  In particular, all accumulation points of trajectories outside of
  $\Theta\times\{0\}$ are contained in $\mathcal{T}_\beta$.

  In the intermediate region $\beta_1<\beta<\beta_2$, as well as for
  the critical parameters, the situation is more intricate and some
  differences appear between forward and pullback dynamics. Again, the
  picture may depend on a $\gamma$-invariant measure in the base
  which serves as a reference. If the cocycle is hyperbolic with
  respect to this measure, a random two-point attractor appears in the
  intermediate parameter regime. This attractor also survives the
  second bifurcation. Consequently, for $\beta>\beta_2$ the forward
  dynamics do not {\em `see'} the whole `torus' $\mathcal{T}_\beta$,
  but only the two-point attractor which is embedded in
  $\mathcal{T}_\beta$. We refer to Theorem~\ref{t.hopf-random} on
  random forcing below for further details.
\item The most important property of the models in (\ref{f_beta}) is
  the fact that the fibre maps send lines passing through the origin
  to such lines again. As a consequence, the map written in polar
  coordinates becomes a double skew product (see Section~\ref{Polar}),
  a fact which will be crucial for our analysis.  Yet, the fact that
  the cocycle $A$ can be chosen arbitrarily allows for a great variety
  of dynamical behaviour when $\beta>\beta_2$. On the one hand, $A$
  could simply be a constant rotation matrix with angle $\rho$. In
  this case $\beta_1=\beta_2$ and the projective action of $A$, which
  is equivalent to the action of $f_\beta$ on $\mathcal{T}_\beta$ is
  typically minimal. On the other hand, we can choose $A$ to be a
  uniformly hyperbolic \sltr-cocycle, which leads to $\beta_1<\beta_2$
  and attractor-repeller dynamics on $\mathcal{T}_\beta$. A mixture of
  these two types occurs when $A$ has non-uniformly hyperbolic
  dynamics and the projective action is minimal (see
  \cite{bjerkloev:2005a} for examples of this type). Then the dynamics
  on $\mathcal{T}_\beta$ are minimal, and thus resemble an irrational
  rotation from the topological point of view, but they are of
  attractor-repeller type from the measurable
  point of view.  \listend
\end{rem}

As indicated by the preceding remark, the second main goal of this
article is to derive a random analogue of
Theorem~\ref{theorem_topological} in the context of random dynamical
systems. The motivation for this is two-fold. First, there is the
obvious intrinsic interest in random forcing processes, which are
modelled in a purely measure-theoretic setting. Secondly, as
mentioned above, even in the topological setting the description of
the typical dynamical behaviour at intermediate or critical
parameters depends on the choice of a reference measure on the base.
Hence, the consideration of measure-preserving driving processes is
required as well, in order to gain a better understanding of
deterministic forcing.
\medskip


In order to model {\em random forcing}, we make the following assumptions.

\alphlist
\item[\bf (R1)] $(\Theta,\mathcal{B},m,\base)$ is a measure preserving
    dynamical system, i.e. $\base:\Theta\to\Theta$ is a
  bi-measurable bijection and $m$ is an ergodic
  $\base$-invariant probability measure;
\item[\bf (R2)] $h:\R^+\to\R^+$ is $\mathcal{ C}^2$, strictly
   increasing, strictly concave, bounded and satisfies
  $h(0)=0$ and $h'(0)=1$.
\item[\bf (R3)] $A:\Theta \to \sltr$ is measurable and bounded.
  \listend

  Since in this setting there is no topological structure on $\Theta$,
  and consequently $\mathcal{A}_\beta$ has no global topological
  structure either, we concentrate on the structure of
  $\mathcal{A}_\beta$ on typical fibres.  Note that
  $\mathcal{A}_\beta$ again has the form given by
  (\ref{e.global-attractor-parametrisation}), where now
  $r_\beta:\Theta\times\R \to \R^+$ is a measurable function which is
  \halb-periodic and upper semi-continuous in the second
  variable.  This time, the bifurcation parameters are determined by
  the Lyapunov exponent of the cocycle $(\gamma,A)$ with respect to
  $m$, which is defined as
\begin{equation}
  \lambda_m(A) \ = \ \lim_{n\to\infty}\frac 1n \int_{\kreis}\log\| A_n(\theta)\| \ dm(\theta) \ .
\end{equation} Note that the limit exists by subadditivity.

Our second main result provides a description of the nonautonomous
Hopf bifurcation in this random setting, where it is also possible to
give more details on the intermediate parameter region. For the
application to the deterministic models we refer to
Remark~\ref{r.random-hopf}(d) below. In contrast to
Theorem~\ref{theorem_topological}, we now provide details on both
forward and pullback dynamics. The reason is that there are
important differences between the two viewpoints, in particular when
$\beta_1\neq\beta_2$.

\begin{thm} \label{t.hopf-random}
 Suppose $(f_\beta)_{\beta\in\R^+}$ is of the form
 (\ref{f_beta}) and satisfies conditions (R1)--(R3). Let
 \[
 \beta_1^m\ :=\ e^{-\lambda_m(A)} \quad \textrm{and} \quad \beta_2^m\ :=\
 e^{\lambda_m(A)} \ .
  \]
  Then there exists a $\gamma$-invariant set $\Theta_0\ssq \Theta$ of
 full measure, such that for all $\theta\in\Theta_0$ the following
 hold.  \alphlist
\item If $\beta< \beta^m_1$, then $\mathcal{A}_\beta(\theta)=\{0\}$
  and
  \begin{equation}
   \nLim f^n_{\beta,\theta}(v) \ = \ 0 \quad  \textrm{ for all } v \in \R^2 \ .
  \end{equation}
\item If $\beta_1^m<\beta< \beta_2^m$, then the set
  $\mathcal{A}_\beta(\theta)$ is a line segment of positive length.
  More precisely, there exist measurable functions
  $\alpha_u,\alpha_s:\Theta_0\to[0,\halb)$, not depending on $\beta$,
  such that $r_\beta(\theta,\alpha)>0$ if and only if
  $\alpha=\alpha_u(\theta)$ and we have
  \begin{equation}
    \mathcal{A}_\beta(\theta) \ = \ \left\{ rv(\alpha_u(\theta)) \mid
      |r| \leq r_\beta(\theta,\alpha_u(\theta)) \right\} \ ,
  \end{equation}
  and the graph of the set-valued function
  \begin{equation} \label{e.Psi} \Psi_\beta(\theta) \ = \ \left\{\pm
      r_\beta(\theta,\alpha_u(\theta))v(\alpha_u(\theta)) \right\}
\end{equation}
is a random two-point forward attractor with domain of attraction
 \begin{equation}\nonumber
  \mathcal{D} \ = \  \{(\theta,v) \mid \theta\in\Theta_0, v\in \R^2
  \smin(\R v(\alpha_s(\theta)) \} \ ,
\end{equation}
in the sense that
 \begin{equation}\label{e.random_attractor}
   \lim_{n\to\infty} d\left(f^n_{\beta,\theta}(v),\Psi_\beta(\gamma^n\theta)\right)
   \ = \ 0 \
  \end{equation}
  for all $(\theta,\alpha)\in\mathcal{D}$.
\item If $\beta > \beta_2^m$, then the map $\alpha\mapsto
  r_\beta(\theta,\alpha)$ is strictly positive and continuous. The set
  $\mathcal{T}_\beta$ defined by
  \begin{equation}
    \mathcal{T}_\beta(\theta)\ = \ \partial \mathcal{A}_\beta(\theta)\ = \
    \left\{ \left. r_\beta(\theta,\alpha)v(\alpha) \
      \right| \alpha\in[0,1) \right\}
\end{equation}
is the global pullback attractor outside $\Theta\times \{0\}$. More
precisely, for all $\delta>0$ there exists an $f_\beta$-forward
invariant random compact set $\mathcal{K}_{\beta,\delta}$ which contains
$\Theta_0\times\left(\overline{B_1(0)}\smin
  B_\delta(0)\right)$ and satisfies
\[
\mathcal{T}_\beta(\theta) \ = \ \ncap
f^n_{\beta,\base^{-n}\theta}\left(\mathcal{K}_{\beta,\delta}(\gamma^{-n}\theta)\right) \ .
\]
When $\lambda_m(A)>0$, the random forward attractor $\Psi_\beta$ given
by (\ref{e.Psi}) still exists, with
$\Psi_\beta(\theta)\ssq\mathcal{T}_\beta(\theta)$, and
(\ref{e.random_attractor}) remains true.  \listend
  \end{thm}

  \begin{rem} \label{r.random-hopf}
    \alphlist
    \item As before, case (b) of the theorem is void if $\beta_1^m=\beta_2^m$.
    \item Note that for $\beta_1^m<\beta<\beta_2^m$, the attractor
      $\Psi_\beta$ given by (\ref{e.Psi}) consists exactly of the endpoints
      of the segment $\mathcal{A}_\beta(\theta)$ on each fibre.
    \item If $\beta_1^m<\beta_2^m$, then the statements on
      $\Psi_\beta$ can be interpreted in the way that this attractor
      persists throughout the whole parameter range (if
      $\beta<\beta_1^m$ it coincides with $\Theta\times \{0\}$ by
      definition) and attracts almost all initial conditions with
      respect to $m$ and the Lebesgue measure on $\R^2$.
    \item When $\base$ is a homeomorphism of a compact metric space
      $\Theta$ as in Theorem~\ref{theorem_topological}, we denote by
      $\mathcal{ M}(\base)$ the set of $\base$-invariant ergodic
      probability measures on $\Theta$. As mentioned, we can apply
      Theorem~\ref{t.hopf-random} and Proposition~\ref{p.critical}
      below for any fixed reference measure $m\in\mathcal{M}(\base)$
      on the base.  As a straightforward consequence of the semiuniform
      sub-multiplicative ergodic theorem (see Theorem~\ref{Semi-uniform
        ergodic}), we have
  \begin{equation}\label{eq:lmax-sup}
  \lmax(A) \ = \ \sup_{m\in\mathcal{M}(\base)} \lambda_m(A) \ .
\end{equation}
Therefore $\beta_1\leq\beta_1^m\leq\beta_2^m\leq\beta_2$.
However, due to compactness of $\mathcal{M}(\base)$, there always
exists at least one $\hat m\in\mathcal{M}(\base)$ with $\lambda_{\hat
  m}(A)=\lmax(A)$ and thus $\beta_1^{\hat m}=\beta_1$ and
$\beta_2^{\hat m}=\beta_2$. When $\beta_1<\beta<\beta_2$, then this
means in particular that $\hat m$-typical fibres are line segments of
positive length and the typical dynamics with respect to $\hat m$ are
governed by a two-point attractor $\Psi_\beta$ given by
(\ref{e.Psi}).  Theorem~\ref{theorem_topological}(b) is a direct
consequence of this.
\item Note that the full measure set $\Theta_0\ssq\Theta$ in the above
  statement is fixed and does not depend on the parameter $\beta$.
  Obtaining this $\beta$-independence will require some additional
  work, but since the parameter set is uncountable this is
  clearly stronger than just showing that all statements hold
  $m$-a.s.\ for all parameters $\beta$, but allowing the exceptional
  set to change with $\beta$. \listend
  \end{rem}

  For the three non-critical parameter regions described above, the
  picture provided by Theorem~\ref{t.hopf-random} can be considered
  rather complete. In contrast to this, the two critical parameters
  $\beta_1^m$ and $\beta_2^m$ are more difficult to treat, and there
  are some questions which we have to leave open here (see
  Questions~\ref{q.critical}). Nevertheless, the following proposition
  provides at least some information, both on pullback and forward
  dynamics.

  \begin{prop} \label{p.critical} Under the assumptions of
    Theorem~\ref{t.hopf-random}, the set $\Theta_0$ can be chosen such
    that for all $\theta\in\Theta_0$ the following hold.
\alphlist
\item If $\beta=\beta^m_1<\beta^m_2$, then $\mathcal{A}_\beta(\theta)=\{0\}$ and
  there exists a set $J=J(\theta)\ssq\N$ of asymptotic density $0$ such that
  \begin{equation}
        \lim_{\substack{n\to\infty \\ n\notin J(\theta)}} \|f^n_{\beta,\theta}(v)\| \ = \ 0
    \quad \textrm{for all } v\in\R^2 \ .
  \end{equation}
\item If $\beta=\beta_1^m=\beta_2^m$, then $\mathcal{A}_\beta(\theta)$ is
  not a topological disk. More precisely, there exists
  $\alpha=\alpha(\theta)$ such that
  $r_\beta(\theta,\alpha(\theta))=0$. Further,
  \begin{equation}
    \nLim \ntel \inergsum \|f^i_{\beta,\theta}(v)\| \ = \ 0 \quad \text{for all } v\in\R^2\ .
  \end{equation}
\item If $\beta_1^m<\beta=\beta_2^m$, then the statement of
  Theorem~\ref{t.hopf-random}(b) holds without any modifications.
\listend
  \end{prop}

  \begin{questions} \label{q.critical}
    \alphlist
  \item In the situation of Theorem~\ref{theorem_topological}, does
    $\mathcal{A}_\beta=\Theta\times\{0\}$ still hold if
    $\beta=\beta_1$? If not, is this always true when
    $\gamma$ is uniquely ergodic?
  \item If the answer to (a) is negative, is it at least true that for
    $\gamma$ uniquely ergodic and $\beta=\beta_1$ we have
    $\mathcal{A}_\beta(\theta)=0$ $m$-a.s.\ and $\nLim\|
    f_{\beta,\theta}^n(v)\| = 0$ for $m$-a.e.\ $\theta$ and all
    $v\in\R^2$?
  \item In the situation of Proposition~\ref{p.critical}(a), is it
    true that $\nLim \|f^n_{\beta,\theta}(v)\|=0$ for $m$-a.e.\
    $\theta\in\Theta$ and all $v\in\R^2$? In other words, does
    Proposition~\ref{p.critical}(a) hold with $J(\theta)=\emptyset$?
  \item In the situation of Proposition~\ref{p.critical}(b), is it
    true that $\mathcal{A}_\beta(\theta)=\{0\}$ $m$-a.s.\ and $\nLim\|
    f_{\beta,\theta}^n(v)\| = 0$ for $m$-a.e.\  $\theta$ and all
    $v\in\R^2$?
\listend
  \end{questions}

  The paper is organised as follows. Section~\ref{Preliminaries}
  provides some basic notation and preliminary results on skew product
  systems with one-dimensional fibres. In Section~\ref{top.doubleskew}
  we introduce a change of coordinates which transforms our system
  into a double skew product. This observation will be crucial for the
  further analysis. The proof of Theorem~\ref{theorem_topological} on
  deterministic forcing is given in Section~\ref{top.setting}, whereas
  Section~\ref{random.setting} deals with the random setting and
  contains the proofs of Theorem~\ref{t.hopf-random} and
  Proposition~\ref{p.critical}. We close with some remarks concerning
  continuous-time systems generated by non-autonomous planar vector fields in
  Section~\ref{CtsModels} and an explicit example illustrated by some
  simulations in Section~\ref{Examples}. \medskip

  \noindent \textbf{Acknowledgements.} V.~Anagnostopoulou and
  T.~J\"ager were supported by the German Research Council
  (Emmy-Noether-Project Ja 1721/2-1), G.~Keller was supported by the
  German Research Council (DFG-grant Ke 514/8-1).

\section{Notation and preliminaries} \label{Preliminaries}

Given a measure-preserving dynamical system {\em (mpds)}
$(\Theta,\mathcal{B},m,\base)$ in the sense of Arnold
\cite{arnold:1998} and a Polish space $M$, we say $f:\Theta\times M\to \Theta\times M$ is a
{\em continuous random map with base $\base$} if it is a measurable
skew product map
\begin{equation}
  \label{e.w-forced} f:\Theta \times M \to \Theta \times M \quad , \quad  (\theta,x)
  \ \mapsto \ (\base\theta,f_\theta(x)) \
\end{equation}
and $x\mapsto f_\theta(x)$ is continuous for all $\theta\in\Theta$.
Note that we write $\gamma\theta$ instead of $\gamma(\theta)$. The
maps $f_\theta :X \to X$ are called {\em fibre maps}.  By
$f^n_\theta=(f^n)_\theta=f_{\gamma^{n-1}\theta}\circ \ldots f_\theta$
we denote the fibre maps of the iterates of $f$ (and not the iterates
of the fibre maps), that is $f^n_\theta(x) = \pi_2\circ
f^n(\theta,x)$. Here $\pi_2:\Theta\times M\to M$ is the projection to
the second coordinate. When $\Theta$ is a metric space and $\base$ is
continuous, such that $f$ is a continuous skew product map, we also
call $f$ a {\em $\base$-forced map}.  When $M$ is a smooth manifold
and all fibre maps $f_\theta$ are $\mathcal{C}^r$, we call $f$ a {\em
  random or $\base$-forced $\mathcal{C}^r$-map}. When $M$ is a real
interval, we say $f$ is a {\em random or $\base$-forced
  $\mathcal{C}^r$-interval map}. If all fibre maps are in addition
(strictly) increasing, we say $f$ is a random of
$\base$-forced {\em monotone} $\mathcal{C}^r$-interval map.

In the context of random maps, fixed points of unperturbed maps are
replaced by {\em invariant graphs}. If $m$ is a $\base$-invariant
measure, then we call a measurable function $\varphi:\Theta\to M$ an
{\em $(f,m)$-invariant graph} if it satisfies
\begin{equation}
  \label{e.inv-graph}
  f_\theta(\varphi(\theta)) \ = \ \varphi(\base\theta) \quad \textrm{for }
  m\textrm{-a.e. } \theta\in\Theta \ .
\end{equation}
When (\ref{e.inv-graph}) holds for all $\theta\in\Theta$, we say
$\varphi$ is an {\em $f$-invariant graph}. However, this notion
usually only makes sense if $\Theta$ is a topological space and
$\varphi$ has some topological property, like continuity or at least
semi-continuity. Note that any $f$-invariant graph is an
$(f,m)$-invariant graph for all $\base$-invariant measures $m$.
Usually, we will only require that $(f,m)$-invariant graphs are
defined $m$-almost surely, which means that implicitly we always speak
of equivalence classes.  Conversely, $f$-invariant graphs are defined
everywhere, and in this case we write
$\graph(\varphi)=\{(\theta,\varphi(\theta))\mid \theta\in\Theta\}$.

The {\em (vertical) Lyapunov exponent} of an $(f,m)$-invariant graph
$\varphi$ is given by
\begin{equation}
  \label{e.lyap}
  \lambda_m(\varphi) \ = \ \int_\Theta \log f'_\theta(\varphi(\theta)) \ d m(\theta) \ .
\end{equation}
In some cases, we will also write $\lambda_m(f,\varphi)$, in order to
avoid ambiguities. Apart from the analogy to fixed points of
unperturbed maps, an important reason for concentrating on invariant
graphs is the fact that there is a one-to-one correspondence between
invariant graphs and invariant ergodic measures of forced monotone
interval maps. If $m$ is a $\base$-invariant ergodic measure and
$\varphi$ is an $(f,m)$-invariant graph, then an $f$-invariant ergodic
measure $m_\varphi$ can be defined by
\begin{equation}
  \label{e.graph-measure}
  m_\varphi(A) \ = \ m\left(\left\{\theta\in\Theta\mid (\theta,\varphi(\theta)) \in A
    \right\}\right) \ .
\end{equation}
Conversely, we have the following.
\begin{thm}[Theorem 1.8.4 in \cite{arnold:1998}] \label{t.measures}
  Suppose $(\Theta,{\cal B},m,\base)$ is an ergodic {\em mpds} and $f$
  is a random monotone $\mathcal{C}^0$-interval map with base $\base$.
  Further, assume that $\mu$ is an $f$-invariant ergodic measure which
  projects to $m$ in the first coordinate. Then $\mu=m_\varphi$ for
  some $(f,m)$-invariant graph $\varphi$.
\end{thm}

Note that any probability measure $\mu$ on $\Theta\times M$ that projects to $m$
can be disintegrated into a family of probability measures
$(\mu_\theta)_{\theta\in\Theta}$ on the fibres, in the sense that
$\int_{\Theta\times M} \Phi \ d\mu =\int_\Theta \int_M \Phi(\theta,x) \
d\mu_\theta(x)dm(\theta)$ for all measurable functions $\Phi:\Theta\times M\to
\R$ \cite[Proposition 1.4.3]{arnold:1998}. Let $\delta_x$ denote the Dirac measure in the point $x$. Then, if
$\mu=m_\varphi$ we obtain $\mu_\theta=\delta_{\varphi(\theta)}$. Consequently,
an ergodic measure associated to an invariant graph can also be called a random
Dirac measure.  Invariant measures associated to $n$-valued invariant graphs are
called random $n$-point measures. Theorem~\ref{t.measures} can then be rephrased
by saying that all ergodic measures of random monotone interval maps are random
Dirac measures.

When the fibre maps of a random monotone $\mathcal{C}^2$-interval map
are all concave, the following result allows to control the number of
invariant graphs and their Lyapunov exponents.

\begin{thm}[\cite{keller:1996}] \label{t.convexity} Suppose
  $(\Theta,{\cal B},m ,\base)$ is a {\em mpds} and $f$ is a
  $\base$-forced monotone ${\cal C}^2$-interval map whose fibre maps
  are all strictly concave. Further, assume that the function
  $\eta(\theta) = \inf_{x\in I(\theta)} \log f'_\theta(x)$ has an
  integrable minorant.  Then there exist at most two $(f,m)$-invariant
  graphs, and if there exist two distinct $(f,m)$-invariant graphs
  $\psi^-\leq \psi^+$ then $\lambda_m(\psi^-)>0$ and
  $\lambda_m(\psi^+)<0$.
\end{thm}
Implicitly, this result is contained in \cite{keller:1996}. A proof
for quasiperiodic forcing can be found in \cite{jaeger:2006a}, which
also remains valid in the more general case stated above.

Another situation where information on the Lyapunov exponent of an
invariant graph is available is the following.
\begin{lemma}
  \label{l.upperboundinglyap} Let $(\Theta,{\cal B},m ,\base)$ be a
{\em mpds} and $f$ be a $\base$-forced monotone ${\cal C}^1$-interval
map with compact fibres $M=[a,b]\ssq \R$. Suppose that the function
$\eta(\theta) = \inf_{x\in M} \log f'_\theta(x)$ has an integrable
minorant and let
\[
\psi^+(\theta) \ = \ \nLim f^n_{\gamma^{-n}\theta}(b) \ .
\]
Then $\psi^+$ is an invariant graph and $\lambda_m(\psi^+)\leq 0$.
\end{lemma}
This result is contained in \cite[Lemma 3.5]{jaeger:2003} for the case
of quasiperiodic forcing, but again the proof given there remains
valid in the more general version stated above.

The following lemma from \cite{anag/jaeger:2011} is a variation of a
result by Sturman and Stark \cite{sturman/stark:2000}.
\begin{lemma}[\cite{anag/jaeger:2011}] \label{l.curve-criterium}
  Suppose $\base$ is a homeomorphism of a compact metric space
  $\Theta$, $f$ is a $\base$-forced ${\cal C}^1$-interval map and $K$
  is a compact $f$-invariant set that intersects every fibre
  $\{\theta\}\times X$ in a single interval. Further, assume that for
  all $\base$-invariant measures $m$ and all $(f,m)$-invariant graphs
  $\psi$ contained in $K$ we have $\lambda_m(\psi)< 0$. Then $K$ is a
  continuous $f$-invariant curve.
\end{lemma}

Now suppose that $T:Y\rightarrow Y$ is a measurable transformation of
a measurable space $Y$ and $(\Phi_n)_{n\in\N}$ is a subadditive
sequence of measureable functions $\Phi_n : Y \to \R$.\footnote{Recall
  that a sequence $\Phi_n:Y \to \R$ is subadditive if
  $\Phi_{m+n}(y)\leq \Phi_n(y)+\Phi_m(T^n(y))$ for all $y\in Y$.}
Let $\mu$ be a $T$-invariant measure and assume that the $\Phi_n$ are
integrable with respect to $\mu$. We write $\mu(\Phi_n)=\int \Phi_n
d\mu$. Then subadditivity yields $\mu(\Phi_{n+m})\leq
\mu(\Phi_{n})+\mu(\Phi_{m})$, and hence Fekete's Subadditivity Lemma
implies that
$$\overline{\Phi}_\mu:=
\lim_{n\rightarrow\infty}\frac{1}{n}\mu(\Phi_n)=\inf_{n\geq0}\frac{1}{n}\mu(\Phi_n)$$
is well defined. In addition, if $\mu$ is ergodic then
$\lim_{n\rightarrow\infty}\frac{1}{n}\Phi_n=\overline{\Phi}_\mu$
$\mu$-almost surely by Kingman's Ergodic Theorem. The following
semi-uniform ergodic theorem from \cite{sturman/stark:2000} will be
used frequently in the discussion of deterministic forcing in
Section~\ref{top.setting}.
\begin{theorem}[Corollary 1.11 in \cite{sturman/stark:2000}]\label{Semi-uniform ergodic}
  Suppose that $T:Y\rightarrow Y$ is a continuous map on a compact
  metrizable space $Y$ and $(\Phi_n)_{n\in\N}$ is a sub-additive
  sequence of continuous functions $\Phi_n : Y \to \R$. Let
  $\lambda\in\R$ be a constant such that $\overline{\Phi}_\mu<\lambda$
  for every $T$-invariant measure $\mu$. Then there exist
  $\varepsilon>0$ and $n_0\in\N$ such that for all $n\geq
  n_0$ we have
$$ \frac{1}{n} \Phi_n(y) \leq \lambda-\varepsilon \qquad
\forall y\in Y.
$$
\end{theorem}

For the case of random forcing, we need a random analogue of this
result. In order to state it, we need some more notation.  Assume
$(\Theta,\mathcal{B})$ is a measurable space, $\gamma:\Theta\to\Theta$
a measurable transformation and $T:\Theta\times M\to \Theta\times M$
is a continuous random map with base $\base$. Given
$m\in\mathcal{M}(\base)$, denote the set of all $T$-invariant
probability measures which project to $m$ by $\mathcal{M}_m(T)$.
Following \cite{arnold:1998,castaing-valadier:1975} we
say $K\ssq \Theta\times M$ is a {\em random compact} set if \romanlist
\item $K(\theta)=\{x\in M\mid (\theta,x)\in K\}$ is compact for
  $m$-a.e.\ $\theta\in\Theta$;
\item the functions $\theta\mapsto d(x,K(\theta))$ are measurable for
  all $x\in M$.  \listend
  $K$ is called {\em forward $T$-invariant} if
  $T_\theta(K(\theta))\ssq K(\base\theta)$ for $m$-a.e.\
  $\theta\in\Theta$.  Given any forward $T$-invariant random compact
  set $K$, we denote the set of $\mu\in\mathcal{M}_m(T)$ which are
  supported on $K$ by $\mathcal{M}^K_m(T)$. Further, we assume that
  $\nfolge{\Phi_n}$ is a subadditive sequence of functions
  $\Phi_n:\Theta\times M\to\R$ which are continuous in the second
  variable and let
\begin{eqnarray*}
   \Phi^{\mathrm{abs}}_n(\theta) & = & \max\{|\Phi_n(\theta,x)|\mid x\in K(\theta)\} \  .
\end{eqnarray*}
We call a random variable $C:\Theta\to \R$ {\em adjusted} with respect
to $\base$, if it satisfies $\nLim \ntel C(\base^n\theta)=0$ for
$m$-a.e.\ $\theta\in\Theta$.
\begin{thm}[\cite{jaeger/keller:2012}]\label{t.random_semiuniform}
  Let $T:\Theta\times M\to \Theta\times M$ be a continuous random map
  over the ergodic {\em mpds} $(\Theta,\mathcal{B},m,\base)$. Suppose
  that $\nfolge{\Phi_n}$ is a subadditive sequence of functions
  $\Phi_n:\Theta\times M\to\R$ which are continuous in the second
  variable. Further, assume that $K$ is a forward $T$-invariant random
  compact set, $\Phi_n^{\mathrm{abs}} \in L^1(m)$ for all $n\in\N$ and
  $\lambda\in\R$ satisfies $\overline\Phi_\mu < \lambda$ for all
  $\mu\in\mathcal{M}^K_m(T)$.  Then there exists $\lambda'<\lambda$
  and a tempered random variable $C:\Theta\to\R$ such that
  \begin{equation}
    \Phi_n(\theta,x) \ \leq \ C(\theta)+n\lambda' \quad \textrm{for } m\textrm{-a.e.}
    \theta\in\Theta \textrm{ and all } x\in K(\theta).
  \end{equation}
  In particular, there exists $\eps>0$ such that for $m$-a.e.\
  $\theta\in\Theta$ there is an integer $n(\theta)\in \N$ with
  \begin{equation}
    \ntel \Phi_n(\theta,x) \ < \ \lambda \qquad \textrm{for all }  n\geq n(\theta)
  \textrm{ and } x\in K(\theta).
  \end{equation}
\end{thm}

\section{Double skew product structure}\label{top.doubleskew}

\subsection{Polar coordinates}\label{Polar}

In order to understand and analyse the dynamics of $f_\beta$, it is
convenient to use projective polar coordinates. Let
$\R^2_*=\R^2\smin\{0\}$ and consider the maps
\begin{eqnarray*}
&p\ : \ \R^2_* \rightarrow \kreis \ , \quad
&p(v)=\frac{1}{\pi}\arctan\left(\frac{v_2}{v_1}\right) \bmod 1\ ,\\
&P\ : \ \R^2_* \rightarrow \kreis \times (0,\infty) \ , \quad
&P(v)=\left(p(v),\|v\|\right)\ .
\end{eqnarray*}
$P$ is two-to-one, and if we let
$$\mathcal{H}^+\ =\ \{(x_1,x_2)\in\R^2_*\mid v_2\geq 0 \textrm{ and } v_2>0
\textrm{ if } v_1<0\}$$ and $\mathcal{H}^-=\R^2_*\smin\mathcal{H}^+$,
then $P$ has two inverse branches
$Q^+=\left(P_{|\mathcal{H}^+}\right)^{-1}$ and
$Q^-=\left(P_{|\mathcal{H}^-}\right)^{-1}$, where $Q^\pm(\alpha,r) =
\pm(r\cos(\pi\alpha),r\sin(\pi\alpha))$.

Let $\hat{P}(\theta,v)=(\theta, P(v))$ and $\hat
Q^\pm(\theta,\alpha,r)=(\theta,Q^\pm(\alpha,r))$. Then the action of $f_\beta$ on
polar coordinates is given by
\begin{equation}
  \label{e.F_betatilde}
  \begin{array}{l}\tilde F_\beta\ : \  \Theta\times\kreis\times(0,\infty)\rightarrow\Theta\times
    \kreis\times(0,\infty)\\ \ \\
    (\theta,\alpha,r)\mapsto\hat{P}\circ
    f_{\beta|\Theta\times R^2_*}\circ
    \hat Q^+(\theta,\alpha,r) \ .
    \end{array}
  \end{equation}
  As $f_{\beta,\theta}(-v)=-f_{\beta,\theta}(v)$, we may equally have used $\hat Q^-$
  instead of $\hat Q^+$. For the same reason $\tilde F_\beta$ is continuous, as the
  discontinuity of $\hat Q^+$ is cancelled by $\hat P$.  We have
\begin{lemma}\label{l.conjugacy}
\begin{enumerate}[label=(\roman*)]
\item $\hat{P}$ is a two-to-one factor map from
  $f_{\beta|\Theta\times\R^2_*}$ to $\tilde F_\beta$.
\item The map $\tilde F_\beta$ extends to an injective skew product map
 \begin{equation}
  \label{e.F_beta}
  \begin{array}{l} F_\beta\ : \  \Theta\times\kreis\times[0,\infty)\rightarrow\Theta\times
    \kreis\times[0,\infty)\\ \ \\
    (\theta,\alpha,r)\mapsto (g(\theta,\alpha),F_{\beta,\theta,\alpha}(r)) \
    \end{array}
  \end{equation}
  such that $(\alpha,r)\mapsto F_{\beta,\theta,\alpha}(r)$ is
  continuous for all $\theta\in\Theta$.
\item The base map $g$ is given by
$$g(\theta,\alpha)=\left(\base\theta,
  g_\theta(\alpha)\right)=\left(\base\theta,\frac{1}{\pi}\arctan\left(
    \frac{c_\theta + d_\theta\tan \pi\alpha}{a_\theta +b_\theta\tan
      \pi\alpha}\right) \bmod 1\right) \ ,
      $$ the fibre maps by
$$
F_{\beta,\theta,\alpha}(r) \ = \ h(\beta r) \Omega(\theta,\alpha) \ ,
$$
where $\Omega(\theta,\alpha) = \|A(\theta)(\cos \pi\alpha, \sin
\pi\alpha)\|$.
\item If $\Theta$ is a metric space and $\base$ is continuous, then
  $F_\beta$ is a continuous map.
\end{enumerate}
\end{lemma}
\begin{proof}
  By the definition of $\tilde F_\beta$, $\hat{P}$ is a two-to-one
  factor map between $\tilde F_\beta$ and
  $f_{\beta|\Theta\times\R^2_*}$. Then,
\begin{eqnarray*}
  F_\beta(\theta,\alpha,r) &=& \hat{P}\left(\base\theta,
    h(\beta
    r)\frac{ A(\theta)}{r}(r\cos \pi\alpha, r\sin\pi \alpha)\right)\nonumber\\
  &=& \left(\base\theta, \frac{1}{\pi}\arctan\left(
      \frac{c_\theta + d_\theta\tan \pi\alpha}{a_\theta +b_\theta\tan
        \pi\alpha}\right)  \bmod
    1,
    h(\beta r)\| A(\theta)(\cos \pi\alpha, \sin
    \pi\alpha)\|\right).
    \end{eqnarray*}
     If we now define
$F_{\beta,\theta,\alpha}(0)=(g(\theta,\alpha),0)$, then
the injectivity and the claimed continuity properties of $F_\beta$ are easy
to verify and the formulae for $g$ and $F_{\beta,\theta,\alpha}$ follow immediately.
\end{proof}

For later use we note that
\begin{equation}\label{eq:Omega_n}
\|A_n(\theta)(\cos\pi\alpha,\sin\pi\alpha)\|
\ = \
\prod_{k=0}^{n-1} \Omega\circ g^k(\theta,\alpha) \ .
\end{equation}

\subsection{Lyapunov exponents}

The above transformation makes it possible to apply existing results
on skew product maps with one-dimensional fibres to study the dynamics
of $F_\beta$ and, subsequently, the dynamics of $f_\beta$. An
important issue in this are the relations between Lyapunov exponents
of the original and the transformed system.  We start with a
corollary of Oseledets's Multiplicative Ergodic Theorem.
\begin{thm}\label{oseledets}
  Let $A: \Theta\rightarrow SL(2,\R)$ be measurable.

  If $\lambda_m(A)>0$, then there exists a splitting
  $\R^2=E^s(\theta)\oplus E^u(\theta)$ such that $A(\theta)\cdot
  E^i(\theta)=E^i(\base\theta)$ for $i=s,u$.  For $i=s,u$ we have
  $E^i(\theta)=\R\cdot v^i(\theta)$ for vectors
  $v^s(\theta),v^u(\theta)\in \R^2$ and
  $\lim_{n\rightarrow\infty}\frac{1}{n}\log\frac{\|A_n(\theta)\cdot
    v\|}{\|v\|}=\lambda_i$ for all non-zero $v\in E^i(\theta)$, where
  $\lambda_u=\lambda_m(A)$ and $\lambda_s=-\lambda_m(A)$.

  If $\lambda_m(A)=0$, then
  $\lim_{n\rightarrow\infty}\frac{1}{n}\log\frac{\|A_n(\theta)\cdot
    v\|}{\|v\|}=0$ for $m$-a.e. $\theta\in\Theta$ and all $v\in \R^2$.
\end{thm}
Now, first assume $\lambda_{m}(A)>0$. Consider the function
$$p\ : \ \R^2_*\rightarrow \kreis\quad ,\quad p(v)\ = \ \frac{1}{\pi}\arctan\left(
  \frac{v_2}{v_1}\right) \bmod 1\ .$$ Then define functions
$\phi_{i}:\Theta\rightarrow\kreis$, by
$\phi_i(\theta)=p(v^i(\theta))$, $i=u,s$, where the
$v^i:\Theta\to\R^2_*$ are as in Theorem \ref{oseledets}. Obviously,
since these graphs correspond to the directions of the invariant
splitting and $g$ is the projective action of the cocycle $(\base,A)$,
they are $(g,m)$-invariant. Further, $\phi_u$ is attracting and $\phi_s$ is
repelling. We summarise these observations in the following folklore
lemma.
\begin{lemma}\label{phi_invgraphs}
  Let $m\in\cM(\base)$ and $\lambda_m(A)>0$. Then the functions
  $\phi_i$, $i=s,u$, are $(g,m)$-invariant graphs and for $m$-a.e.\
  $\theta\in\Theta$ and all $\alpha\neq \phi_s(\theta)$ we have
  \begin{equation}\label{e.graph-attraction}
    \nLim d\left(g^n_\theta(\alpha),\phi_u(\base^n\theta)\right) \ = \ 0 \ ,
  \end{equation}
  with exponential speed of convergence.  In particular, no other
  $(g,m)$-invariant graphs except $\phi_u$ and $\phi_s$ exist.
\end{lemma}

Furthermore, the associated random Dirac measures
$m_{\phi_i}(B):=m(\{\theta\in\Theta \mid (\theta,\phi_i(\theta))\in
B)$, $i=s,u$, are the only $g$-invariant and ergodic measures which
project to $m$. This follows from an old result by Furstenberg. In
order to state it, we denote the set of $g$-invariant ergodic measures
which project to $m\in\mathcal{M}(\base)$ by $\mathcal{M}_m(g)$.
\begin{lemma}[Furstenberg \cite{furstenberg:1961}]\label{lemmaB} Suppose
  $(\Theta,\mathcal{B},m,\base)$ is an ergodic {\em mpds} and $g
  :\Theta\times\kreis\to\Theta\times\kreis$ is a random map whose
  fibre maps $g_\theta$ are all circle homeomorphisms. Then, if there
  exists a $(g,m)$-invariant graph, all $\mu\in\cM_m(g)$ are of the
  form $\mu=m_\phi$ for some $(g,m)$-invariant graph $\phi$.
\end{lemma}
\medskip

The crucial observation of this section is the following.

\begin{proposition}\label{mainlyapunovlemma} Suppose $(f_\beta)_{\beta\in[0,1]}$
  satisfies (D1)--(D3). Let $\beta_1=e^{-\lmax}$ and $\beta_2=e^{\lmax}$.
  Then
\begin{eqnarray*}
  \beta_1 &= & \sup\left\{\beta\in\R^+ \mid  \lambda_\mu(F_\beta,0)<0\
  \forall \ \mu\in\cM(g)\right\}\quad \textrm{ and}\\
  \beta_2& = & \inf\left\{\beta\in\R^+\mid
  \lambda_\mu(F_\beta,0)>0\ \forall \ \mu\in\cM(g)\right\} \ ,
\end{eqnarray*}
 where
  $\lambda_\mu(F_\beta,0)=\int_{\Theta\times\kreis} \log
  F_{\beta,\theta,\alpha}'(0)d\mu(\theta,\alpha)$.
\end{proposition}
\smallskip

In order to prove this, we show the following more general statement.
Note that when $\lambda_m(A)>0$, then Lemma~\ref{phi_invgraphs} and
the subsequent remark imply $\cM_m(g)=\{m_{\phi_u},m_{\phi_s}\}$.
\begin{lemma}\label{lemmaA} Let $\mu\in \cM_m(g)$. If $\lambda_m(A)>0$, then
  $\lambda_\mu(F_\beta,0)=\lambda_m(A)+\log\beta$ if $\mu=m_{\phi_u}$
  and $\lambda_\mu(F_\beta,0)=-\lambda_m(A)+\log\beta$ if
  $\mu=m_{\phi_s}$. If $\lambda_m(A)=0$, then
  $\lambda_\mu(F_\beta,0)=\log\beta$.
\end{lemma}
\begin{proof}
  First, let $\lambda_m(A)>0$. By Lemmas~\ref{phi_invgraphs} and~\ref{lemmaB}, $\mu=m_{\phi_i}$
  with $i\in\{s,u\}$. Fix $\theta\in\Theta$ and let
  $\alpha=p(v^i(\theta))$, where $v^u(\theta)$ and $v^s(\theta)$ are
  chosen as in Theorem~\ref{oseledets}. We have that
  $F_{\beta,\theta,\alpha}'(r)=\beta h'(\beta r) \Omega(\theta,\alpha)$ and
  thus $F_{\beta,\theta,\alpha}'(0)=\beta \Omega(\theta,\alpha)$. Hence
\begin{eqnarray}\label{eqn_lemmaA}
  \lambda_\mu(F_\beta,0)&=& \int_{\Theta\times\kreis}\log(\beta\cdot \Omega(\theta,\alpha))
  d\mu(\theta,\alpha)=\log\beta+\int_{\Theta\times\kreis}\log \Omega(\theta,\alpha) d\mu(\theta,\alpha)
  \nonumber\\
  &=& \log\beta +\int_\Theta\log \Omega(\theta,\phi_i(\theta)) dm(\theta) \   .
\end{eqnarray}
Now, by Theorem \ref{oseledets} and equation ({\ref{eq:Omega_n}}),
\begin{eqnarray}
\lambda_i &=& \lim_{n\rightarrow\infty}\frac{1}{n}\log
\frac{\|A_n(\theta) v^i(\theta)\|}{\|v^i(\theta)\|} \ = \ \lim_{n\rightarrow\infty}\frac{1}{n}\log \nonumber
  \prod_{k=0}^{n-1}\Omega\circ g^k(\theta,\phi^i(\theta)) \\ & = &
  \lim_{n\rightarrow\infty}\frac{1}{n}\sum_{k=0}^{n-1}\log \Omega(\base^k\theta,\phi_i(\base^k\theta)))\
= \  \int
  \log \Omega(\theta,\phi_i(\theta))d m(\theta)
  \end{eqnarray}
%
for $m$-a.e. $\theta\in\Theta$ by Birkhoff's Ergodic Theorem. Therefore, equation
(\ref{eqn_lemmaA}) becomes:
$$\lambda_\mu(F_\beta,0)=\log\beta+\lambda_i=\left\{ \begin{array}{ll}
 \log\beta+\lambda_m(A) & \textnormal{for } i=u\\
 \log\beta-\lambda_m(A) & \textnormal{for } i=s
 \end{array}\right.$$
which proves the lemma for $\lambda_m(A)>0$.\smallskip

Now, assume $\lambda_m(A)=0$. For $\mu$-a.e.\ $(\theta,\alpha)$ we
have
$$\lambda_\mu(F_\beta,0)\ =\ \lim_{n\rightarrow\infty}\frac{1}{n}\log(F_{\beta,\theta,\alpha}^n)'(0)\ .$$
Let $v=(\cos(\pi\alpha),\sin(\pi\alpha))$ and
$\Omega(\theta,v)=\frac{\|A(\theta)v\|}{\|v\|}$. Similar to above, we
obtain
\begin{eqnarray*}
  \lefteqn{\nLim \ntel\log\left(F^n_{\beta,\theta,\alpha}\right)'(0) \ = \ \nLim \ntel \knergsum
    \log F'_{\beta,g^k(\theta,\alpha)}(0) } \\
  & = & \log\beta + \nLim \ntel \knergsum \log b\circ g^k(\theta,\alpha) \ = \
  \log\beta + \nLim \ntel\log\frac{\|A_n(\theta)v\|}{\|v\|} \ = \ \log \beta \ .
\end{eqnarray*}
Hence, $\lambda_\mu(F_\beta,0) = \log\beta$ as claimed.
\end{proof}
We are now in position to prove Proposition \ref{mainlyapunovlemma}.
\begin{proof}[Proof of Proposition \ref{mainlyapunovlemma}]
%
%
Observing Lemma~\ref{lemmaA} and equation~(\ref{eq:lmax-sup}), we have
\begin{eqnarray*}
\sup\{\beta\in\R^+\mid \lambda_\mu(F_\beta,0)<0\
\forall \ \mu\in\cM(g)\}
&=&
\sup\{\beta\in\R^+\mid \log\beta + \lambda_m(A)<0\
\forall \ m\in\cM(\base)\}\\
&=&
\sup\{\beta\in\R^+\mid \log\beta + \lambda_{\max}(A)<0\}\\
&=&
e^{-\lambda_{\max}(A)}=\beta_1\ .
\end{eqnarray*}
\end{proof}

\section{Deterministic forcing: Proof of
  Theorem~\ref{theorem_topological}}\label{top.setting}

We first analyse the skew product system $F_\beta$ in the two
parameter regimes $\beta<\beta_1$ and $\beta>\beta_2$. Application
of the results to the original system $f_\beta$ will then be
straightforward. As mentioned in Remark~\ref{r.random-hopf}(d),
statement (b) of Theorem~\ref{theorem_topological} on the
intermediate parameter region $\beta_1<\beta<\beta_2$ is a direct
consequence of the results on random forcing, such that we do not
need to consider this case here.

Throughout this section, we assume that $(f_\beta)_{\beta\in\R^+}$ satisfies
(D1)--(D3). In particular, $\Theta$ is a compact metric space and
$\gamma:\Theta\to\Theta$ is a homeomorphism. Since due to (\ref{e.scaling}) we
have $F_{\beta,\theta,\alpha}(1)\leq 1$ for all $\beta,\theta$ and $\alpha$, the
global attractor of $F_\beta$ is given by
\begin{equation}\label{e.doubleskewattractor}
\tilde{\mathcal{A}}_\beta \ = \ \ncap F^n_\beta\left(\Theta\times\kreis\times[0,1]\right) \ .
\end{equation}
Due to the monotonicity of the fibre maps $F_{\beta,\theta,\alpha}$, an
invariant graph $\psi^+_\beta$ can be defined as
\begin{equation}\label{e.upperboundinggraph}
\psi^+_\beta(\theta,\alpha) \ = \ \sup \tilde{\mathcal{A}}_\beta(\theta,\alpha)
\ = \ \nLim F^n_{\beta,g^{-n}(\theta,\alpha)}(1) \ .
\end{equation}
We call $\psi^+_\beta$ the {\em upper bounding graph} of $F_\beta$. Note that
$\tilde{\mathcal{A}}_\beta = \left\{(\theta,\alpha,r) \mid
r\in\left[0,\psi^+_\beta(\theta,\alpha)\right]\right\}$. Independent of $\beta$, a second
invariant graph is always given by $\psi^-(\theta,\alpha)=0$. Depending on
$\beta$, we may or may not have $\psi^+_\beta=\psi^-$. \medskip

 {\bf The case $\beta<\beta_1$.} This is
the simpler of the two cases, where, as it will be shown, $\psi^-=\psi^+_\beta$
is the only invariant graph of the system.

\begin{prop}\label{p.doubleskew.case1}
  Suppose $\beta < \beta_1$.  Then the global attractor
  $\tilde{\mathcal{A}}_\beta$ is equal to
  $\Theta\times\kreis\times\{0\}$. In particular, $\psi^-$ is the
  unique invariant graph of the system, all invariant measures are
  supported on $\Theta\times\kreis\times\{0\}$ and
  \[
  \lim_{n\rightarrow\infty}F^n_{\beta,\theta,\alpha}(r)\ =\ 0\quad \textrm{for
    all } (\theta,\alpha,r)\in\Theta\times\kreis\times[0,\infty) \ .
  \]
\end{prop}
\proof From the concavity of the fibre maps $F_{\beta,\theta,\alpha}$ and
  the Mean Value Theorem we obtain that $F_{\beta,\theta,\alpha}^n(r)\leq
  (F_{\beta,\theta,\alpha}^n)'(0) \cdot r$ for all $r\in\R^+$. We claim that
  $(F_{\beta,\theta,\alpha}^n)'(0)\rightarrow 0$ as
  $n\rightarrow\infty$. Consider the additive sequence of continuous functions
  $\Phi_n:\Theta\times\kreis\rightarrow \R$, defined by
  \[
  \Phi_n(\theta,\alpha)\ =\ \sum_{i=0}^{n-1}\log
  F_{\beta,g^i(\theta,\alpha)}'(0)\ =\ \log(F^n_{\beta,\theta,\alpha})'(0) \ .
  \]
  $\Phi_n$ satisfies the assumptions of Theorem~\ref{Semi-uniform
    ergodic} for $T=g$ and $\lambda=0$, since
  $\int_{\Theta\times\kreis} \Phi_1\ d\mu=\lambda(F_\beta,0) <0$ for
  all $\mu\in\cM(g)$ by Proposition \ref{mainlyapunovlemma}.
  Hence, there exists $\varepsilon>0$ and $n_0\in\N$ such that for all
  $n\geq n_0$ and all $(\theta,\alpha)\in\Theta\times\kreis$ we have
  $\frac{1}{n}\log (F_{\beta,\theta,\alpha}^n)'(0) \leq -\varepsilon$,
  that is, $(F_{\beta,\theta,\alpha}^n)'(0) \leq e^{-\varepsilon\cdot
    n}\rightarrow 0$ as $n\rightarrow\infty$. As this convergence is
  uniform in $\theta$ and $\alpha$, the statements of the
  proposition follow immediately.  \qed\medskip

\noindent {\bf The case $\beta>\beta_2$.} Here, the aim is to prove the
continuity and strict positivity of $\psi^+_\beta$, whose preimage under the
projection $\hat P$ then defines the split-off torus for the original system
$f_\beta$. We start with an auxiliary lemma.

\begin{lemma}\label{l.doubleskew}
There exists $\delta_0>0$ such that for all
$0<\delta<\delta_0$ there exists an $F_\beta$-forward invariant
compact set $K$ with
$$\Theta\times\kreis\times[\delta,1]\ \subseteq \ K
\subseteq\Theta\times\kreis\times(0,1] \ , $$ and such that
$K(\theta,\alpha)=\{r\in\R^+\mid (\theta,\alpha,r)\in K\}$ is an
interval for all $\theta\in\Theta$.
\end{lemma}
\begin{proof}

  The function $\psi^-(\theta,\alpha)=0$ is invariant graph for
  $F_\beta$. Since $\beta>\beta_2$, we have that
  $\lambda_\mu(\psi^-)=\lambda_\mu(F_\beta,0)>0$ for all $\mu\in\cM
  (g)$ by Proposition \ref{mainlyapunovlemma}. Hence, as the set
  $\Theta\times\kreis\times\{0\}$ is compact and invariant under
  $F_\beta$, Theorem~\ref{Semi-uniform ergodic} applied to $T=g$,
  $\Phi_n(\theta,\alpha)=-\log\left(F^n_{\beta,\theta,\alpha}\right)'(0)$
  and $\lambda=0$ implies that for some $\eps>0$ and $n_0\in\N$ we
  have $\log\left(F^n_{\beta,\theta,\alpha}\right)'(0)>n\eps$ for all
  $\theta\in\Theta$, $\alpha\in\kreis$ and $n\geq n_0$.  Thus, the set
  $\Theta\times\kreis\times\{0\}$ is uniformly repelling for
  $F_\beta^n$ in the vertical direction.

  Let $D:=\Theta\times\kreis\times[\delta,1]$ for some $\delta>0$.  We
  claim that when $\delta$ is sufficiently small, this set is forward
  invariant under $F_\beta^n$ for large $n$. More precisely, there
  exists $n_0\in\N$ such that
\begin{equation}\label{eqn_forwardinv}
F_\beta^n(D)\subseteq D \quad \forall n\geq n_0.
\end{equation}

The uniform repulsion of $\Theta\times\kreis\times\{0\}$ implies
that for all $n\geq n_0$, there exist $\delta (n)>0$ such that
$$\log\left(
  F_{\beta,\theta,\alpha}^n\right)'(r)\ > \ \frac{\varepsilon}{2}\ > \ 0 \quad
\forall \ (\theta,\alpha,r)\in
\Theta\times\kreis\times[0,\delta(n)].$$ Now let $\delta_0=\min\{
\delta(n_0),\ldots, \delta(2n_0-1)\}$ and $\delta\in(0,\delta_0)$. We have that
$$F_{\beta,\theta,\alpha}^n(r)\ \geq F_{\beta,\theta,\alpha}^n(\delta)\ \geq \  \delta\cdot
\left(F_{\beta,\theta,\alpha}^n\right)'(\delta) \ \geq \ \delta$$ for
all $\delta\in[0,\delta_0]$, $n\in\{n_0,\ldots,2n_0-1\}$ and
$(\theta,\alpha,r)\in \Theta\times\kreis\times[\delta,1]$.
Inductively, $F_{\beta,\theta,\alpha}^n(r)\geq \delta$ for all $n\geq
n_0$ and all $(\theta,\alpha,r)\in\Theta\times\kreis\times[\delta,1]$.
This proves claim (\ref{eqn_forwardinv}).

We define the set $K:=\bigcup_{m=0}^{{n_0}-1}F_\beta^m(D)$. Then
$$F_\beta(K)=\bigcup_{m=1}^{n_0}F_\beta^m(D)=F_\beta^{n_0}(D)\cup\left( \bigcup_{m=1}^{n_0-1}F_\beta^m(D)
\right)\subseteq D\cup\left( \bigcup_{m=1}^{n_0-1}F_\beta^m(D)\right)=K.$$
Therefore $K$ is compact and $F_\beta$-forward invariant, and
clearly $\Theta\times\kreis\times[\delta,1]=D\subseteq K
\subseteq\Theta\times\kreis\times(0,1]$. If $K(\theta,\alpha)$ is not
an interval for all $(\theta,\alpha)\in \Theta\times\kreis$, then we
can replace $K$ with the set $\tilde K=\{(\theta,\alpha,r)\mid \exists
r_1,r_2\in K(\theta,\alpha): r_1\leq r\leq r_2\}$.  Due to the
monotonicity of the fibre maps $F_{\beta,\theta,\alpha}$, this set
$\tilde K$ is still $F_\beta$-forward invariant and therefore has all
the required properties. \end{proof}


\begin{proposition}\label{p.doubleskew.case3}
  The set $\hat{K}:=\bigcap_{n\in\N}F^n_\beta(K)$ is $F_\beta$-invariant and
  equals $\graph(\psi^+_\beta)$. In particular,
  $\psi^+_\beta:\Theta\times\kreis\rightarrow[0,1]$ is continuous and strictly
  positive.
\end{proposition}
\proof As $K$ is compact and $F_\beta$-forward invariant, $\hat K$
is compact and $F_\beta$-invariant.  By Theorem \ref{t.measures},
all $F_\beta$-invariant measures are of the form $\nu=\mu_{\psi}$
for some $\mu\in\cM(g)$ and some $(F_\beta,\mu)$-invariant graph
$\psi$. The graph $\psi^-(\theta,\alpha)=0$ is always invariant and
$\lambda_\mu(\psi^-)=\lambda_\mu(F_\beta,0)>0$ for all
$\mu\in\cM(g)$. Therefore, Theorem \ref{t.convexity} yields that for
all $\mu\in\cM(g)$ the only other possible $(F_\beta,\mu)$-invariant
graph is $\psi^+_\beta$, which is $\mu$-a.s. strictly positive.


  Hence $\psi=\psi^+_\beta$, and again by Theorem~\ref{t.convexity} we have
  $\lambda_\mu(\psi^+_\beta)<0$. Thus, Lemma \ref{l.curve-criterium} yields
  that $\hat{K}$ is a continuous $F_\beta$-invariant curve. Consequently,
  $\psi^+_\beta$ has to be the unique continuous function
  $\Theta\times\kreis\rightarrow\R^+$ such that
  $\hat{K}=\graph(\psi^+_\beta)$.
\qed\medskip

\begin{cor}
  The upper bounding graph $\psi^+_\beta$ is attracting, in the sense that
$$\lim_{n\rightarrow\infty}(F^n_{\beta,\theta,\alpha}(r)-\psi^+_\beta(g^n(\theta,\alpha)))=0
\qquad \forall \
(\theta,\alpha,r)\in\Theta\times\kreis\times(0,\infty).$$
\end{cor}
\begin{proof}
  Due to the definition of $\psi^+_\beta$ and the monotonicity of the fibre
  maps, it is enough to show that for all $\delta \in (0,\delta_0)$,
  with $\delta_0$ from Lemma~\ref{l.doubleskew}, we have
  $\lim_{n\rightarrow\infty}(F_{\beta,\theta,\alpha}^n(1)-F_{\beta,\theta,\alpha}^n(\delta))=0$
  for all $(\theta,\alpha)\in\Theta\times\kreis$. Fixing
  $\delta<\delta_0$ and choosing $K$ as in Lemma~\ref{l.doubleskew},
  we have that
\begin{equation}\label{cor.doubleskew_eq} F_{\beta,\theta,\alpha}^n(1)-F_{\beta,\theta,\alpha}^n(\delta)
  \leq (1-\delta)\cdot\sup_{(\theta,\alpha,r)\in
    K}(F_{\beta,\theta,\alpha}^n)'(r)\end{equation} Consider
$\Phi_n(\theta,\alpha,r)=\log(F_{\beta,\theta,\alpha}^n)'(r)$.
$\Phi_n$ is an additive sequence and by invoking
Theorem~\ref{t.measures} and \ref{t.convexity} as in the preceding
proof, we obtain that $\overline{\Phi}_{\tilde{\nu}}<0$ for all
measures $\tilde{\nu}\in\cM(F_\beta)$ supported on $K$. Thus, by
Theorem \ref{Semi-uniform ergodic} there exists $\varepsilon>0$ and
$n_0\in\N$ such that for all $n\geq n_0$ we have
$\frac{1}{n}\log(F_{\beta,\theta,\alpha}^n)'(r) \leq -\varepsilon$,
which means that $(F_{\beta,\theta,\alpha}^n)'(r)\leq
e^{-\varepsilon\cdot
  n}$. Consequently, $\sup_{(\theta,\alpha,r)\in
  K}(F_{\beta,\theta,\alpha}^n)'(r)\leq e^{-\varepsilon \cdot
  n}\nKonv 0$, which completes the proof.
\end{proof}

\proof[\bf Proof of Theorem~\ref{theorem_topological}.] Since
\begin{equation} \label{e.attractor-relation}
\mathcal{A}_\beta\ = \ \hat P^{-1}\left(\tilde{\mathcal{A}}_\beta\smin
  \Theta\times\kreis\times\{0\}\right)
\end{equation}
and $\|f^n_{\beta,\theta}(v)\| = F^n_{\beta,\theta,p(v)}(r)$, statement (a) of
the theorem follows immediately from
Proposition~\ref{p.doubleskew.case1}. Further, as mentioned in
Remark~\ref{r.random-hopf}(d), statement (b) is a direct consequence of
Theorem~\ref{t.hopf-random}, whose proof is independent of
Theorem~\ref{theorem_topological}.

It remains to prove statement (c) on the parameters $\beta>\beta_2$.
However, due to (\ref{e.attractor-relation}) this follows directly
from Proposition~\ref{p.doubleskew.case3} and
Corollary~\ref{cor.doubleskew_eq}. Note that, thus,
$\mathcal{T}_\beta=\hat
P^{-1}\left(\graph\left(\psi^+_\beta\right)\right)$ and
$r_\beta(\theta,\alpha)=\psi^+_\beta(\theta,2\alpha\bmod 1)$.
\qed\medskip

\section{Random forcing}\label{random.setting}

Throughout this section, we assume that $(f_\beta)_{\beta\in\R^+}$
satisfies (R1)--(R3). In particular $(\Theta,\mathcal{B},m)$ is a
probability space and $\gamma:\Theta\to\Theta$ is a measure-preserving
bijection. As before, the global attractor is given by
(\ref{e.doubleskewattractor}) and we have
\begin{equation}
  \label{e.boundingequation}
  \tilde{\mathcal{A}}_\beta \ = \ \left\{(\theta,\alpha,r)\in
   \Theta\times\kreis\times\R^+ \mid 0\leq r \leq \psi^+_\beta(\theta,\alpha)\right\} \ ,
\end{equation}
where the upper bounding graph $\psi^+_\beta$ is given by
(\ref{e.upperboundinggraph}) as before.  We start again by analysing
the double skew product $F_\beta$ in the different parameter regimes
$\beta<\beta_2$, $\beta_1<\beta<\beta_2$ and $\beta>\beta_2$, and
then apply the results to the original system $f_\beta$. In each
case, we have to take particular care to ensure that the exceptional
set of measure zero in the statements can be chosen independent of
the parameter $\beta$.

\subsection{The non-critical parameter regions: Proof of Theorem~\ref{t.hopf-random}}

{\bf The case $\beta<\beta_1^m$.} Again, this is the simplest case,
where the global attractor equals $\Theta\times\kreis\times\{0\}$.
\begin{lemma} \label{l.random-attraction}
There exists a set
  $\Theta_1\ssq\Theta$ of full measure, such that for all $\beta <
  \beta_1^m$ and all $\theta\in\Theta_1$ we have
  $\tilde{\mathcal{A}}_\beta(\theta) = \kreis\times\{0\}$ and
  \begin{equation} \label{e.random_attraction}
    \lim_{n\rightarrow\infty}F^n_{\beta,\theta,\alpha}(r)\ =\ 0\quad
    \textrm{for all } (\alpha,r)\in\kreis\times[0,\infty) \ .
  \end{equation}
\end{lemma}
\proof Since
$\tilde{\mathcal{A}}_\beta(\theta)
  = \kreis\times\{0\}$ is equivalent to
$\psi^+_\beta(\theta,\alpha)=0$ for all $\alpha\in\kreis$, we first
want to show that
\begin{equation} \label{e.random_attraction2}
  \psi^+_\beta(\theta,\alpha)\ = \ 0 \quad \textrm{for } m\textrm{-a.e.}\
  \theta\in\Theta \textrm{ and all } \alpha\in\kreis . \
\end{equation}
However, we have
\begin{eqnarray*}
  \psi^+_\beta(\theta,\alpha) & = & \nLim F^n_{\beta,g^{-n}(\theta,\alpha)}(1) \\
  & \leq & \limsup_{n\to\infty} \left( F^n_{\beta,g^{-n}(\theta,\alpha)}\right)'(0) \\
  & \leq & \limsup_{n\to\infty} \beta^n\cdot \|A_n(\base^{-n}\theta)\| \ = \ 0 \quad m\textrm{-a.s.}
\end{eqnarray*}
since $\nLim \ntel \log \|A_n(\base^{-n}\theta)\| = \nLim \log
\|A_{-n}(\theta)\| = \lmax(A)$ $m$-a.s. by Kingmans Subadditive
Ergodic Theorem and $\beta< \beta_1^m=e^{-\lmax(A)}$.

The fact that \eqref{e.random_attraction} also holds $m$-a.s.\ is
proved exactly in the same way, replacing the pullback iteration by
forward iteration. Hence, for every fixed $\beta<\beta_1^m$, the set
\[
\Theta(\beta) \ = \ \left\{\theta\in\Theta\mid
\psi_\beta^+(\theta,\alpha)=0\ \textrm{ and } \nLim
  F^n_{\beta,\theta,\alpha}(r)=0 \ \forall \alpha,r\right\}
\]
has full measure. However, since the fibre maps
$F_{\beta,\theta,\alpha}$ and the upper bounding graph $\psi^+_\beta$
are increasing in $\beta$, the set $\Theta(\beta)$ is decreasing in
$\beta$. Therefore
\[
\Theta_1 \ := \ \bigcap_{\beta<\beta_1}\Theta(\beta) \ = \
\bigcap_{n\in\N} \Theta((1-1/n)\beta_1^m)
\]
has full measure and satisfies the assertions of the lemma.  \qed
\bigskip

\noindent {\bf The case $\beta_1^m<\beta<\beta_2^m$.} In this case, we
split the proof into several lemmas.
\begin{lemma} \label{l.intermediate_upperbounding}
  There exists a set $\Theta_2\ssq\Theta$ of full measure such that
  for all $\beta\in(\beta_1^m,\beta_2^m)$ and $\theta\in\Theta_2$ we have
  \[
  \psi^+_\beta(\theta,\phi_u(\theta))\ > \ 0 \quad \textrm{and} \quad
  \psi^+_\beta(\theta,\alpha)\ = \ 0 \quad \textrm{for all }
  \alpha\in\kreis\smin\{\phi_u(\theta)\} \ .
  \]
\end{lemma}
\proof For the inverse action $g^{-1}$, which is the projective
action of the inverse cocycle $(\base^{-1},A^{-1})$, the roles of
$\phi_u$ and $\phi_s$ exchange, and $\phi_s$ becomes the attractor.
Hence, due to Lemma~\ref{phi_invgraphs} we have that for $m$-a.e.\
$\theta$ and all $\alpha\neq \phi_u(\theta)$
\[
\nLim d(g^{-n}_\theta(\alpha),\phi_s(\base^{-n}\theta)) \ = \ 0 \ .
\]
As a consequence, we obtain that
\begin{eqnarray*}
  \lefteqn{  \nLim \ntel\log\left(F^n_{\beta,g^{-n}(\theta,\alpha)}\right)'(0)
    \ = \ \nLim \ntel \log \left(F^n_{\beta,\base^{-n}\theta,\phi_s(\base^{-n}\theta)}\right)'(0) }\\
  & = & \nLim-\ntel\log\left(F^{-n}_{\beta,\theta,\phi_s(\theta)}\right)'(0)\ = \
   \lambda_{m_{\phi_s}}(F_\beta,0) \
  = \ - \lambda_m(A)+\log\beta \ < \ 0 \ .
\end{eqnarray*}
Therefore
\begin{eqnarray*}
  \psi^+_\beta(\theta,\alpha) & = & \nLim F^n_{\beta,g^{-n}(\theta,\alpha)}(1) \\
  &\leq & \nLim\left(F^n_{\beta,g^{-n}(\theta,\alpha)}\right)'(0) \ =
  \ \nLim\left(\beta e^{-\lambda_m(A)}\right)^n \ = \ 0 \ .
\end{eqnarray*}
On the other hand, we have
$\lambda_{m_{\phi_u}}(\psi^-)=\lambda_m(A)+\log\beta >0$. By
Lemma~\ref{l.upperboundinglyap}, we therefore have
$\psi^+_\beta(\theta,\alpha)\neq\psi^-(\theta,\alpha)$
$m_{\phi_u}$-a.s., which means that
$\psi^+_\beta(\theta,\phi_u(\theta))>0$
$m$-a.s.~. Finally, using the monotonicity of $\psi_\beta^+$ in
$\beta$ as in the proof of Lemma~\ref{l.random-attraction}, it is easy
to check that the exceptional set of measure zero in all the
statements can be chosen independent of $\beta$. For this, we have to
use that the set of $\theta$ where $\psi^+_\beta(\theta,\alpha)=0$ for
all $\alpha\neq\phi_u(\theta)$ is decreasing in $\beta$, whereas the
set of $\theta$ with $\psi^+_\beta(\theta,\phi_u(\theta))>0$ is
increasing with $\beta$. \qed\medskip

Now, let
\begin{equation}
  \psi^+_{\beta,u}(\theta) \ = \
\left(\phi_u(\theta),\psi^+_\beta(\theta,\phi_u(\theta))\right) \ .
\end{equation}
We have
$F_\beta(\theta,\psi^+_{\beta,u}(\theta))=(\gamma\theta,\psi^+_{\beta,u}(\theta))$
$m$-a.s., such that $\psi^+_{\beta,u}$ is an $(F_\beta,m)$-invariant
graph when $F_\beta$ is viewed as a skew product with base $\gamma$
and two-dimensional fibres $\kreis\times\R^+$. In order to show that
$\psi^+_{\beta,u}$ is a random attractor with domain of attraction
$\tilde{\mathcal{D}}=\{(\theta,\alpha,r)\mid \alpha\neq
\phi_s(\theta),\ r\neq 0\}$, which is the equivalent to
(\ref{e.random_attractor}), we first need some preliminary statements.
We start by fixing some more notation.

Given $\theta,\alpha$ and $r$, we let
$(\theta_n,\alpha_n,r_n)=F^n_\beta(\theta,\alpha,r)$ and
\[
\Omega_n(\theta,\alpha) \ = \ \prod_{i=1}^{n-1} \Omega\circ g^i(\theta,\alpha)
\ = \
\|A_n(\theta)(\cos\pi\alpha,\sin\pi\alpha)\|
\ ,
\]
see equation (\ref{eq:Omega_n}).
Further, given $a,b\in(0,1]$, we let
\[
\Gamma(a,b) \ = \ \inf_{a\leqslant x\leqslant 1} \inf_{0< q
  \leqslant b} \frac{h(qx)}{qh(x)} \ .
\]
\begin{lemma} \label{l.gammaab}
  $\Gamma(a,b)\geq 1$ and $\Gamma(a,b)>1$ if $b<1$.
\end{lemma}
\begin{proof}
  As $h$ is strictly concave, $\frac{h(qx)}{qh(x)}>1$ for each $x>0$
  and each $q\in(0,1)$. Furthermore, for $x_n\in[a,1]$ and $q_n\to0$,
\begin{equation}
\liminf_{n\to\infty}\frac{h(q_nx_n)}{q_nh(x_n)} \ = \
\liminf_{n\to\infty}\frac{x_nh'(0)}{h(x_n)}\
\geq\
\frac{h'(0)}{h(a)/a}\ > \ 1 \ .
\end{equation}
Now the claim follows from continuity of $h$ and compactness of $[a,1]\times[0,b]$.
\end{proof}

The next statement allows to compare orbits with the same
$\theta$-coordinate.
\begin{lemma}[Forward comparison lemma]\label{lemma:main-comparison}
Let $\alpha,\alpha'\in\kreis$ and $r,r'\in\R^+\smin\{0\}$ and set
$q_k:=\frac{r_k'}{r_k}$ and
$\hat q_k:=\min\{q_k,1\}$.
If $k<n$ and $q_{k+1},\dots,q_{n-1}\leq 1$, then
\begin{equation}\label{eq:q_n-estimate}
  q_n\   \geq \
  \hat q_k\cdot
  \frac{\Omega_{n-k}(\theta_k,\alpha_k')}{\Omega_{n-k}(\theta_k,\alpha_k)} \cdot
  \prod_{j=k+1}^{n-1}\Gamma(\beta r_j, q_j) \ .
\end{equation}
\end{lemma}
\begin{proof} We have
\begin{equation}
\frac{h(\beta r_j')}{h(\beta r_j)}\
\geq \
\frac{h(\hat q_j\beta r_j)}{h(\beta r_j)}
\ \geq \
\hat q_j\cdot \Gamma(\beta r_j,\hat q_j)
\end{equation}
so that
\begin{equation}
  q_{n} \  = \
     \frac{\Omega(\theta_{n-1},\alpha_{n-1}')}{\Omega(\theta_{n-1},\alpha_{n-1})}
  \cdot  \frac{h(\beta r_{n-1}')}{h(\beta r_{n-1})}
  \ \geq \
  \hat q_{n-1}
  \cdot \frac{\Omega(\theta_{n-1},\alpha_{n-1}')}{\Omega(\theta_{n-1},\alpha_{n-1})}
    \cdot  \Gamma(\beta r_{n-1},\hat q_{n-1}) \ .
\end{equation}
If $q_{k+1},\dots,q_{n-1}\leq 1$, then $q_j=\hat q_j$ for
$j=k+1,\dots,n-1$, and an easy induction yields
\begin{equation}\label{eq:q_n-estimate-proof}
\begin{split}
  q_n \ &\geq \ \hat q_k\cdot \prod_{j=k}^{n-1}
  \frac{\Omega(\theta_j,\alpha_j')}{\Omega(\theta_j,\alpha_j)} \cdot
  \prod_{j=k}^{n-1}\Gamma(\beta r_j,\hat q_j)
  \\
   & \geq \ \hat q_k\cdot
  \frac{\Omega_{n-k}(\theta_k,\alpha_k')}{\Omega_{n-k}(\theta_k,\alpha_k)} \cdot
   \prod_{j=k+1}^{n-1}\Gamma(\beta r_j,q_j) \ .
\end{split}
\end{equation}
\end{proof}

We can now turn to the attractor property of $\psi^+_{\beta,u}$.

\begin{lemma}\label{lemma:replaceL5.8}
  Suppose $\lambda_m(A)>0$. Then there exists a set
  $\Theta_3\ssq\Theta$ of full measure such that for all
  $\beta>\beta_1^m$ and $\theta\in\Theta_3$ we have
  $\psi^+_\beta(\theta,\phi_u(\theta))>0$ and
     \[
  \nLim
  d\left(F^n_{\beta,\theta}(\alpha,r),\psi^+_{\beta,u}
      \left(\gamma^n\theta\right)\right)  = 0 \quad
      \textrm{ for all }\alpha\in\kreis\smin\{\phi^s(\theta)\},\ r>0\ .
    \]
    In particular, for the skew product $F_\beta$ with base $\base$
    the invariant graph $\psi^+_{\beta,u}$ is a random one-point
    attractor with domain of attraction
    $\tilde{\mathcal{D}}=\left\{(\theta,\alpha,r)\in\Theta_3\times\kreis\times\R^+\mid
      \alpha\neq \phi_s(\theta),\ r> 0\right\}$.
\end{lemma}

\begin{proof}
We will prove the equivalent assertion
\begin{equation}
  \nLim
  d\left(\left(g_{\beta,\theta}^n(\alpha),F^n_{\beta,\theta,\alpha}(r)\right),
    \psi^+_{\beta,u}(\base^n\theta)\right) \ = \ 0  \  .
\end{equation}
In view of Lemma~\ref{phi_invgraphs},
$d(g_{\beta,\theta}^n(\alpha),\phi_u(\base^n\theta))$ tends to zero
exponentially fast for all $\theta$ in a set $\Lambda_1\ssq \Theta$ of
full measure.  So it remains to estimate
$d\left(F_{\beta,\theta,\alpha}^n(r),
  \psi^+_{\beta}(\base^n\theta,\phi_u(\base^n\theta))\right)$.

Let $\mu=m_{\phi^u}$. If $\beta>\beta_1^m$, then
$\lambda_\mu(F_\beta,0)=\lambda_m(A)+\log\beta>0$, so that
$\psi^+_\beta(\theta,\phi_u(\theta))>0$ for $m$-a.e.\ $\theta$. Since
$\psi_\beta^+$ is monotonically increasing in $\beta$, we can fix a
\base-invariant set $\Lambda_2\ssq \Lambda_1$ of full measure such that
\[
\psi_\beta^+(\theta,\phi_u(\theta))\ > \ 0 \quad \textrm{for all }
\beta>\beta_1^m,\ \theta\in\Lambda_2 \ .
\]
Now, let $\beta>\beta_1^m$, $\theta\in\Lambda_2,\ \alpha=\phi_u(\theta)$
and $r=\psi^+_\beta(\theta,\alpha)$ and choose $\alpha'\in\kreis$ with
$\alpha'\neq \phi_s(\theta)$ and $r'>0$.  We have to prove that $\nLim
|r_n'-r_n|=0$, which will follow from the stronger assertion
$\nLim\frac{r_n'}{r_n}=1$. (Observe that since
$F_{\beta,\theta,\alpha}(\R^+)\ssq[0,1]$, we may assume without loss
of generality that $r_n$ and $r_n'$ are bounded by 1.)

We first show $\liminf_{n\to\infty} \frac{r_n'}{r_n}\geq 1$.  Let
$\delta>0$. For given $n\in\N$ let
\begin{equation}\label{eq:ell(n)}
\ell\ = \ \ell(n)\ = \
\begin{cases}
\max\left\{k\leq n:
r_k'\geq(1-\delta)r_k\right\}&\text{if such a $k$ exists}\\
0&\text{otherwise.}
\end{cases}
\end{equation}
If $\ell=n$, then $\frac{r_n'}{r_n}\geqslant1-\delta$. For $\ell<n$,
Lemma~\ref{lemma:main-comparison} implies
\begin{equation}\label{eq:v_n-quot}
  \frac{r_n'}{r_n} \
  \geq \
  \min\left\{
    \frac{r_\ell'}{r_\ell},1\right\}\cdot \frac{\Omega_{n-\ell}(\theta_\ell,
    \alpha_\ell')}{\Omega_{n-\ell}(\theta_\ell,\alpha_\ell)} \cdot
  \prod_{j=\ell+1}^{n-1}\Gamma(\beta r_j, q_j) \ .
\end{equation}

Assume for a contradiction that there is $\ell_0\in\N$ such that
$\ell(n)=\ell_0$ for infinitely many~$n$. Then
\begin{equation}\label{eq:ass-for-contra}
  \limsup_{n\to\infty}\frac{r_n'}{r_n}
  \ \geq \
  \min\left\{
    \frac{r_{\ell_0}'}{r_{\ell_0}},1\right\}
  \cdot\limsup_{n\to\infty}
  \prod_{j=\ell_0+1}^{n-1}\Gamma\left(\beta\,\psi^+_\beta(\base^j\theta,
  \phi_u(\base^j\theta)), q_j\right)\cdot
  \kappa(\theta,\ell_0)
\end{equation}
where $\kappa(\theta,\ell_0)=\nLim \Omega_{n-\ell_0}(\theta_{\ell_0},
\alpha_{\ell_0}') /\Omega_{n-{\ell_0}}(\theta_{\ell_0},\alpha_{\ell_0}
) >0$ is the coefficient of the unstable direction in the unique
decomposition of
$v(\alpha'_{\ell_0})=(\cos\pi\alpha'_{\ell_0},\sin\pi\alpha'_{\ell_0})$
with respect to the Oseledets splitting at $\theta_{\ell_0}$.  As
$\psi^+_\beta(\theta,\phi^u(\theta))>0$ $m$-a.s. and as
$q_j=\frac{r'_j}{r_j}\leq 1-\delta$ for $j>\ell(n)=\ell_0$ due to
(\ref{eq:ell(n)}), Lemma~\ref{l.gammaab} implies that this product
diverges as $n\to\infty$ for $m$-a.e. $\theta$. Moreover, since
$\psi_\beta^+$ is monotonically increasing in $\beta$, so is the
product. For this reason, we can fix a set $\Lambda_3\ssq\Lambda_2$ of
full measure such that the product diverges for all
$\theta\in\Lambda_3$ and $\beta>\beta_1^m$.

However, this divergence contradicts the fact that
$r_n'<(1-\delta)r_n$. Hence, if $\theta\in\Lambda_3$ then
$\ell(n)\to\infty$ as $n\to\infty$. As $d(\alpha_j',\alpha_j)\to0$
exponentially fast, which also means
$|\Omega(\theta_j,\alpha_j')-\Omega(\theta_j,\alpha_j)|\to 0$ exponentially
fast, we obtain
\begin{equation}
  \nLim
  \frac{\Omega_{n-\ell(n)}(\theta_{\ell(n)},\alpha_{\ell(n)}')}{\Omega_{n-\ell(n)}(\theta_{\ell(n)},\alpha_{\ell(n)})} \
  = \
  \nLim
  \prod_{j=\ell(n)}^{n-1}
  \frac{\Omega(\theta_j,\alpha_j')}{\Omega(\theta_j,\alpha_j)} \ = \ 1 \ .
  \end{equation}
Therefore (\ref{eq:v_n-quot}) and Lemma~\ref{l.gammaab} imply that
\begin{equation}
  \liminf_{n\to\infty}\frac{r_n'}{r_n}
  \ \geq \ (1-\delta) \ .
\end{equation}
As this is true for each $\delta>0$, we have indeed that
$\liminf_{n\to\infty} \frac{r_n'}{r_n}\geq 1$.

In order to show that $\limsup\frac{r_n'}{r_n} = \left(\liminf
  \frac{r_n}{r_n'}\right)^{-1} \leq 1$, we can apply essentially the
same reasoning with interchanged roles of $(\alpha,r)$ and
$(\alpha',r')$. The only difference is that the
product in (\ref{eq:ass-for-contra}) is replaced by
$\prod_{j=\ell_0+1}^{n-1}\Gamma\left(\beta r_j', q_j\right)$. But
in view of the estimate proved above,
$r_n'\geq\frac{1}{2}r_n$ for all sufficiently large $n$
so that
\begin{equation}
  \prod_{j=\ell_0+1}^{n-1}\Gamma\left(\beta r_j', q_j\right) \
  \geq \
  \prod_{j=\ell_0+1}^{n-1}\Gamma\left(\frac{\beta}{2}r_j, q_j\right)
  =
  \prod_{j=\ell_0+1}^{n-1}\Gamma\left(\frac{\beta}{2}\psi^+_\beta
  (\base^j\theta,\phi_u(\base^j\theta)), q_j\right).
\end{equation}
Again, this product diverges for all $\theta$ in a set
$\Theta_3\ssq\Lambda_3$ of full measure and we conclude that
$\ell(n)\to\infty$. Similar to before, this yields that
$\liminf_{n\to\infty}\frac{r_n}{r_n'}\geq 1$. This proves $\nLim
\frac{r_n'}{r_n}=1$, and thus completes the proof.
\end{proof}

\noindent {\bf The case $\beta>\beta_2^m$.} The following lemma
describes the detachment of the attractor from the central manifold
$\Theta\times\kreis\times\{0\}$ of the double skew product system.

\begin{lem} \label{l.random-detachment} If $\beta>\beta_2^m$, then there
  exists a random variable
  $\Delta_\beta:\Theta\times\kreis\to(0,1]$
  such that
 \begin{equation}\label{e.gamma-radius}
   F_{\beta,\theta,\alpha}(\Delta_\beta(\theta,\alpha)) \ \geq \
   \Delta_\beta(g(\theta,\alpha)) \quad\textrm{ for all } \theta\in\Theta,\
  \alpha\in\kreis \ .
\end{equation}
Furthermore, the set
\begin{equation}
  \Theta(\beta) \ = \ \left\{ \theta\in\Theta\mid \alpha\mapsto\Delta_\beta(\theta,\alpha)
  \textrm{ is continuous and strictly positive} \right\}
\end{equation}
has full measure.
\end{lem}
\proof Suppose $\beta>\beta_2^m$. For some $\delta>0$ specified below,
we define $\Delta_\beta$ for all $(\theta,\alpha)\in\Theta\times\kreis$ by
\[
\Delta_{\beta,\delta}(\theta,\alpha) \ := \ \inf_{n\geq 0}
F^n_{\beta,g^{-n}(\theta,\alpha)}(\delta) \ .
\]
Then
\begin{eqnarray*}
  \Delta_{\beta,\delta}(g(\theta,\alpha)) & =
  & \inf_{n\geq 0} F^n_{\beta,g^{-(n-1)}(\theta,\alpha)}(\delta) \\
  & = & \min\left\{\delta,\inf_{n\geq 1}F_{\beta,\theta,\alpha}\left(
      F^{n-1}_{\beta,g^{-(n-1)}(\theta,\alpha)}(\delta)\right)\right\} \\
  & = & \min\left\{\delta,F_{\beta,\theta,\alpha}(\Delta_{\beta,\delta}(\theta,\alpha))\right\} \ \leq \
  F_{\beta,\theta,\alpha}(\Delta_{\beta,\delta}(\theta,\alpha)) \ .
\end{eqnarray*}
It remains to show that for sufficiently small $\delta>0$ the set
$\Theta(\beta)$ has full measure.  To that end, let
\[
\Phi^\eta_n(\theta) \ = \ \inf_{\alpha\in\kreis} \log
\left(F^n_{\beta,\base^{-n}\theta,\alpha}\right)'(\eta)\ .
\]
Note that since $h$ is concave and $A:\Theta\to\sltr$ is bounded, we
have uniform and monotonically increasing convergence
\begin{eqnarray*}
  \Phi^\eta_n(\theta) & \stackrel{\eta\to 0}{\longrightarrow} &
  \Phi_n(\theta):= \inf_{\alpha\in\kreis} \log
  \left(F^n_{\beta,\base^{-n}\theta,\alpha}\right)'(0) \\ &= &
  -\log\left\|Df_{\beta,\base^{-n}\theta}^n(0)^{-1}\right\| \ = \ n\log\beta -
  \log\left\|A_{-n}(\theta)\right\| \ .
\end{eqnarray*}
on $\Theta$ for all $n\in\N$. As $A$ is an \sltr-cocycle, forward and
backward Lyapunov exponent coincide and we have
\[
\inf_{n\in\N} \ntel \int_{\Theta} \Phi_n \ dm \ = \ \log\beta -
\lambda_m(A) \ =: \ \tilde\lambda \ .
\]
Since $\beta>\beta_2^m$ we have $\tilde\lambda>0$, and we can fix
$k\in\N$ with $\int_\theta \Phi_k\ dm > \tilde\lambda/2$. By choosing
$\eta>0$ sufficiently small we can further ensure that
\[
\int_\Theta\Phi_k^\eta \ dm \ > \ \tilde\lambda/2 \ > \ 0
\]
as well. Since $\base$ is ergodic, Birkhoff's Ergodic Theorem implies
that for $m$-a.e.\ $\theta\in\Theta$ we can choose an integer
$n(\theta)$ such that for all $n\geq n(\theta)$ we have
\begin{equation} \label{e.phik-birkhoff} \sum_{i=0}^{n-1}
  \Phi^\eta_k\circ \base^{-n+i}(\theta) \ > \ n\tilde \lambda/2 +
  2k\sup_{\theta\in\Theta} \Phi^\eta_k(\theta) \ .
\end{equation}
If $n\geq n(\theta)$ and $m$ is the largest integer such that
$mk\leq n$, then this implies $\sum_{i=k}^{mk-1} \Phi^\eta_k\circ
\base^{-n+i}(\theta)\geq mk\tilde\lambda/2$, and consequently there
exists at least one $j\in\{0\ld k-1\}$ such that
\begin{equation}
\sum_{i=1}^{m-1} \Phi^\eta_k\circ
\base^{-n+ik+j}(\theta)\ \geq \ m\tilde\lambda/2 \ .
\end{equation}
If $F^i_{\beta,g^{-n}(\theta,\alpha)}(\delta) \leq \eta$ for all
$i=0\ld n-1$, then due to the concavity of the fibre maps we obtain
\begin{eqnarray*}
  \log\left(F^n_{\beta,g^{-n}(\theta,\alpha)}\right)'(\delta) &
  \geq & m\tilde\lambda/2 \  + \
  \log \left(F^j_{\beta,g^{-n}(\theta,\alpha)}\right)'(\delta) \\ & &  + \
  \log \left(F^{n-(m-1)k-j}_{\beta,g^{-n+(m-1)k+j}(\theta,\alpha)}\right)'
\left(F^{(m-1)k+j}_{\beta,g^{-n}(\theta,\alpha)}(\delta)\right)
\end{eqnarray*}
By choosing $n(\theta)$ large
enough and using the fact that $\log F_{\beta,\theta,\alpha}'(r)$ is uniformly
bounded on \mbox{$\Theta\times\kreis\times[0,\eta]$}, we can
therefore ensure the following:\smallskip
\begin{equation}
  \label{e.0line_expansion}
  \textit{If } n\geq n(\theta) \textit{ and }
  F^i_{\beta,g^{-n}(\theta,\alpha)}(\delta) \leq \eta \textit{
    for all } i=0\ld n-1, \textit{ then }
  \left(F^n_{\beta,g^{-n}(\theta,\alpha)}\right)'(\delta) \ > \ 1 \ .
  \end{equation}

Now, choose any $\delta\leq\eta$ and let
\[
\hat \kappa_\theta(\alpha) \ = \ \min_{n=0}^{n(\theta)}
F^n_{\beta,g^{-n}(\theta,\alpha)}(\delta) \ .
\]
Since the minimum is taken over a finite number of continuous and strictly
positive curves, $\hat \kappa_\theta$ is continuous and strictly positive as
well. Hence, in order to prove the lemma it suffices to show that
$\Delta_\beta(\theta,\alpha) = \hat\kappa_\theta(\alpha)$.

In order to see this, suppose $n\geq n(\theta)$. We proceed by induction on $n$
to show that in this case
\begin{equation} \label{e.gamma_hat}
  F^n_{\beta,g^{-n}(\theta,\alpha)}(\delta) \ \geq \ \hat\kappa_\theta(\alpha) \ .
\end{equation}
First, suppose that $F^j_{\beta,g^{-n}(\theta,\alpha)}(\delta)\geq
\delta$ for some $j\in\{1\ld n\}$. Then by induction assumption we
obtain $F^n_{\beta,g^{-n}(\theta,\alpha)}(\delta)\geq
F^{n-j}_{\beta,g^{-n+j}(\theta,\alpha)}(\delta)\geq \hat \kappa_\theta(\alpha)$.
Otherwise, we can apply (\ref{e.0line_expansion}) and the concavity of
the fibre maps to obtain $F^n_{\beta,g^{-n}(\theta,\alpha)}(\delta)
\geq \delta\geq \hat\kappa_\theta(\alpha)$. Hence, (\ref{e.gamma_hat})
holds in both cases, and this shows $\hat\kappa_\theta(\alpha) =
\Delta_\beta(\theta,\alpha)$ as required.  \qed\medskip

We now turn to the existence and the attractor property of the
invariant torus. Here, particular attention is required to guarantee
the $\beta$-independence of the exceptional set of measure zero.

\begin{lem}
  \label{l.ds-attractingtorus} There exists a $\gamma$-invariant set
  $\Theta_4\ssq \Theta$ of full measure such that for all
  $\beta>\beta_2^m$ there exists a random variable
  $\rho_\beta:\Theta\times\kreis\to[0,1]$ with the following
  properties.  \romanlist
\item For all $\theta\in\Theta_4$ the mapping
  $\alpha\mapsto\rho_\beta(\theta,\alpha)$ is strictly positive and
  continuous.
  \item For all $\delta>0$ and $\theta\in\Theta$ we have
  \begin{equation} \label{e.rhobeta_forwardinvariance}
  F_{\beta,\theta,\alpha}(\delta\rho_\beta(\theta,\alpha)) \ \geq \
  \delta\rho_\beta(g(\theta,\alpha)) \ .
  \end{equation}
  In particular, $K_{\beta,\delta}=\{(\theta,\alpha,r)\mid
  \theta\in\Theta,\ \alpha\in\kreis, r\in
  [\delta\rho_\beta(\theta,\alpha),1]\}$ is an $F_\beta$-forward
  invariant random compact set.
  \item The random compact set
   \begin{equation}\label{e.Sbeta-def}
     \mathcal{S}_\beta\ = \ \bigcap_{n\in\N} F_\beta^n(K_{\beta,\delta}) \
   \end{equation}
   is $F_\beta$-invariant and for all $\theta\in\Theta_4$ we have
   \begin{equation} \label{e.Sbeta-fibres} \mathcal{S}_\beta(\theta) \
     = \ \left\{(\alpha,\psi^+_\beta(\theta,\alpha))\mid
       \alpha\in\kreis\right\} \ .
\end{equation}
In particular $\alpha\mapsto\psi^+_\beta(\theta,\alpha)$ is
continuous.
\listend
\end{lem}
\proof Let $B_n = [\beta^-_n,\beta^+_n]$ be a nested sequence of
intervals with $B_n\nearrow (\beta_2^m,+\infty)$ as $n\to\infty$.
Further, choose a sequence $\delta_k\in(0,1)$ with $\kLim \delta_n=0$.
We define
\[
\rho_\beta \ = \ \Delta_{\beta^-_n} \quad \textrm{ for all } \beta\in
B_n\smin B_{n-1} \ ,
\]
where $\Delta_{\beta^-_n}$ is the random variable from
Lemma~\ref{l.random-detachment}. Note that thus
$\alpha\mapsto\rho_\beta(\theta,\alpha)$ is strictly positive and
continuous for all $\beta>\beta_2$ and $\theta\in \bigcap_{n\in\N}
\Theta(\beta^-_n)=:\widehat\Theta$, where the sets
$\Theta(\beta^-_n)$ are again taken from
Lemma~\ref{l.random-detachment}. Note that $\widehat \Theta$ has
full measure. Further, (\ref{e.rhobeta_forwardinvariance}) holds for
$\beta=\beta^-_n$ and $\delta=1$, and by concavity and monotonicity
of the fibre maps in $\beta$, it extends to all $\beta\in B_n$ and
$\delta\in(0,1]$. The fact that $K_{\beta,\delta}$ is
$F_\beta$-forward invariant is then obvious, thus we have shown (i)
and (ii).

Since the $K_{\beta,\delta}$ are forward invariant, the set
$\mathcal{S}_\beta$ is the nested intersection of random compact sets
and therefore randomly compact itself. Hence, the crucial point is to
show that the fibres $\mathcal{S}_\beta(\theta,\alpha)=\{r\in\R^+\mid
(\theta,\alpha,r)\in\mathcal{S}_\beta\}$ consist of the single point
$\psi^+_\beta(\theta,\alpha)$. Then $\mathcal{S}_\beta(\theta)$ equals
the graph of $\alpha\mapsto\psi^+_\beta(\theta,\alpha)$, and since
$\mathcal{S}_\beta(\theta)$ is compact this implies the continuity of
$\alpha\mapsto\psi^+_\beta(\theta,\alpha)$.

We let
\[
a_n^{\beta,\delta}(\theta,\alpha) =
F^n_{\beta,g^{-n}(\theta,\alpha)}(\delta\rho_\beta(\theta,\alpha)) \quad \textrm{ and
}\quad b_n^\beta(\theta,\alpha) = F^n_{\beta,g^{-n}(\theta,\alpha)}(1)
\ .
\]
Note that by definition $\nLim
b_n^\beta(\theta,\alpha)=\psi^+_\beta(\theta,\alpha)$. Moreover,
\begin{equation}\label{e.K-image}
  \left(F^n_\beta(K_{\beta,\delta})\right)(\theta) \ = \ \left\{(\alpha,r)\mid
  a_n^{\beta,\delta}(\theta,\alpha)\leq r\leq b_n^\beta(\theta,\alpha) \right\} \ .
\end{equation}
Hence, it suffices to show that on a set of full measure and for all
$\beta>\beta_2^m$ we have
\[
\nLim \sup_{\alpha\in\kreis}\left( b^\beta_n(\theta,\alpha) -
a^{\beta,\delta}_n(\theta,\alpha)\right) \ = \ 0 \ .
\]
In order to do so, we fix $\ell$ and $k$ and consider the extended system
\[
\widehat F(\theta,\alpha,r,\beta) \ = \ (F_\beta(\theta,\alpha,r),\beta)
\]
defined on $\Theta\times\kreis\times\R^+\times\R^+$. As
$K_{\beta^-_l,\delta_k}$ is $F_\beta$-forward invariant for all
$\beta\in B_l$, the set $\mathcal{K}_{l,k} =
K_{\beta^-_l,\delta_k}\times B_l$ is forward invariant under $\widehat
F$.

Since the action of $\widehat F$ on $\beta$ is the identity, any
$\widehat F$-invariant ergodic measure which projects to $m$ in the
first coordinate and is supported on $\mathcal{K}_{l,k}$ is a direct
product $\nu\times\delta_\beta$, where $\delta_\beta$ is a Dirac
measure in $\beta\in B_l$ and $\nu\in\mathcal{M}_m(F_\beta)$ is
supported on $K_{\beta^-_l,\delta_k}$.  Further, by
\mbox{Theorem~\ref{t.measures}}, all $F_\beta$-invariant measures
$\nu\in\mathcal{M}_m(F_\beta)$ are of the form $\nu=\mu_\psi$ for some
$\mu\in\mathcal{M}_m(g)$ and some $(F_\beta,\mu)$-invariant graph
$\psi$. Since $\psi^-=0$ is always invariant,
Theorem~\ref{t.convexity} yields that there exists at most one
$(F_\beta,\mu)$-invariant graph $\psi$ which is strictly positive, and
we have $\lambda_\mu(\psi)<0$.  As a consequence, the additive
sequence of functions
\[
\Phi_n(\theta,\alpha,r,\beta) \ = \ \log
\left(F^n_{\beta,g^{-n}(\theta,\alpha)})\right)'\left(F^{-n}_{\beta,\theta,\alpha}(r)\right)
\ = \ -\log\left(F^{-n}_{\beta,\theta,\alpha}\right)'(r)
\]
satisfies the assumptions of Theorem~\ref{t.random_semiuniform} with
$\gamma$ replaced by $\gamma^{-1}$ and $\lambda=0$. Note that $\int
\Phi_1\ d\nu\times\delta_\beta = \int \log F_{\beta,\theta,\alpha}'(r)
\ d\nu(\theta,\alpha,r) = \lambda_\mu(\psi)<0$ for all
$\nu\in\mathcal{M}_m(F_\beta)$ supported on $K_{\beta^-_l,\delta_k}$.

Hence, there exist $\lambda'<0$ and a set $\Theta_{l,k}$ of full
measure, such that for all $\theta\in\Theta_{l,k}$ there exists
$n(\theta)\in\N$ with
\begin{equation}\label{e.random-upperbound}
  \sup\left\{\left.\log\left(F^n_{\beta,\base^{-n}\theta,\alpha}\right)'(r)
      \ \right|\ \beta\in B_l,\ \alpha\in\kreis,\ r\in K_{\beta^-_l,\delta_k}
     (\base^{-n}\theta) \right\} \
  < \ -n\lambda'\quad \textrm{ for all } n\geq n(\theta) \ .
\end{equation}
For fixed $\theta\in\Theta_{l,k}$ and all $\beta\in B_l$ and $n\geq
n(\theta)$ we therefore obtain
\[
\sup_{\alpha\in\kreis}
\left(b_n^\beta(\theta,\alpha)-a_n^{\beta,\delta_k}(\theta,\alpha)\right)
\ \leq \ e^{-n\lambda'} \ \nKonv 0 \ .
\]
In particular, we have
\[
\psi^+_\beta(\theta,\alpha) \ = \ \nLim a_n^{\beta,\delta_k}(\theta,\alpha) \ = \nLim
b_n^\beta(\theta,\alpha) \ ,
\]
as required, and the convergence is even uniform in $\alpha\in\kreis$
and $\beta\in B_l$. If we now define $\Theta_4=\widehat\Theta\cap
\bigcap_{l,k\in\N}\Theta_{l,k}$, then (i)-(iii) hold for all
$\beta>\beta_2^m$, $\delta>0$ and $\theta\in\Theta_4$. \qed\medskip

\proof[\bf Proof of Theorem~\ref{t.hopf-random}.]  As in
Section~\ref{top.setting}, the translation of the above results to the
original setting is now straightforward. We let
$\Theta_0=\Theta_1\cap\Theta_2\cap\Theta_3\cap\Theta_4$, where the
full measure sets $\Theta_1,\Theta_2,\Theta_3$ and $\Theta_4$ are
taken from Lemmas \ref{l.random-attraction},
\ref{l.intermediate_upperbounding}, \ref{lemma:replaceL5.8} and
\ref{l.ds-attractingtorus}, respectively. Then, using the facts that
\begin{itemize}
\item the projection $\hat P$ conjugates $f_{\beta|\Theta\times\R^2_*}$ with
$F_{\beta|\Theta\times\kreis\times(0,\infty)}$,
\item $\mathcal{A}_\beta=\hat P^{-1}\left(\tilde{\mathcal{A}}_\beta\smin
  \Theta\times\kreis\times\{0\}\right) \cup \left(\Theta\times\{0\}\right)$,
\item $\|f^n_{\beta,\theta}(v)\| = F^n_{\beta,\theta,\alpha}(r)$ when
  $v=\pm(r\cos(\pi\alpha),r\sin(\pi\alpha))$,
\item $\Psi_\beta(\theta) =
  P^{-1}\left(\left\{\psi^+_{\beta,u}(\theta)\right\}\right)$ for
  $\beta>\beta_1^m$ and
\item $\mathcal{T}_\beta=\hat P^{-1}(\mathcal{S}_\beta)$ for
  $\beta>\beta_2^m$,
\end{itemize}
statement (a) of Theorem~\ref{r.random-hopf} follows from
Lemma~\ref{l.random-attraction}, (b) follows from
Lemmas~\ref{l.intermediate_upperbounding} and \ref{lemma:replaceL5.8}
and (c) follows from Lemmas~\ref{lemma:replaceL5.8} and
\ref{l.ds-attractingtorus}. In (b) and (c),
$r_\beta(\theta,\alpha)=\psi^+_\beta(\theta,2\alpha\bmod 1)$. The random
set $\mathcal{K}_{\beta,\delta}$ in (c) is defined as
\[
\mathcal{K}_{\beta,\delta} \ =
\left\{(\theta,(r\cos(2\pi\alpha),r\sin(2\pi\alpha))^t)\mid
  \theta\in\Theta,\ \alpha\in\kreis,\
  r\in[\delta\rho_\beta(\theta,2\alpha\bmod 1),1]\right\} \ ,
\]
where $\rho_\beta$ is the random variable from
Lemma~\ref{l.ds-attractingtorus}. \qed\medskip

\begin{rem}
  We want to close this section with a remark on an alternative proof
  of Lemma~\ref{l.ds-attractingtorus}. In the above argument, the
  uniform contraction in the fibres of the set $K_{\beta,\delta}$ is
  obtained by applying Theorem~\ref{t.random_semiuniform}, in
  combination with Theorem~\ref{t.convexity} to guarantee the negativity
  of the vertical Lyapunov exponents in $K_{\beta,\delta}$.

  In the particular situation we consider, it is also possible to give
  a direct proof, without invoking these two general results, by
  making stronger use of the strict concavity of the fibre maps. The
  crucial observation for this is the fact that any orbit which
  frequently stays further than a fixed distance away from the zero
  line. In order to see this, let
  \[
  q_\beta(r) \ = \ \frac{h(\beta r)/r}{\beta h'(\beta r)} \
  \]
  and note that for all $\theta$ and $\alpha$ we have
  \[
  \frac{F_{\beta,\theta,\alpha}(r)/r}{F_{\beta,\theta,\alpha}'(r)} \ = \ q_\beta(r) \ .
  \]
  Due to the strict concavity of $h$, the function $q_\beta$ satisfies $q_\beta(r) > 1 =
  \lim_{\rho\to 0}q_\beta(\rho)$ for all $r>0$.
  Furthermore, we have
  \begin{eqnarray*}
    F^n_{\beta,\theta,\alpha}(r) & = & r\cdot \prod_{i=0}^{n-1}\frac{r_{i+1}}{r_i} \
    = \   r \cdot  \prod_{i=0}^{n-1}
    \frac{F_{\beta,\theta_i,\alpha_i}(r_i)}{r_i} \\
    & = & r \cdot \prod_{i=0}^{n-1}
    q_\beta(r_i) F_{\beta,\theta_i,\alpha_i}'(r_i)  \ = \ r \cdot \left(F^n_{\beta,\theta,\alpha}\right)'(r)
    \cdot  \prod_{i=0}^{n-1}  q_\beta(r_i)
  \end{eqnarray*}
  where we used the notation
  $(\theta_i,\alpha_i,r_i)=F^i_{\beta}(\theta,\alpha,r)$.  Since the
  fibre maps are all bounded by 1, we obtain that
  \begin{equation}\label{e.X}
  \log\left(F^n_{\beta,\theta,\alpha}\right)'(r) \ \leq \ -\log r -
  \inergsum \log q_\beta(r_i) \ .
  \end{equation}
    Now, let $\bar\rho_\beta(\theta) =
  \inf_{\alpha\in\kreis}\rho_\beta(\theta,\alpha)$, where
  $\rho_\beta=\Delta_{\beta^-_n}$ is defined as in the proof of
  Lemma~\ref{l.ds-attractingtorus}. By
  Lemma~\ref{l.random-detachment}, $\bar\rho_\beta(\theta)$ is
  strictly positive for all $\theta\in\Theta(\beta^-_n)$,
    so the same is true for $\log q_\beta(\bar\rho_\beta(\theta))$. In
   addition, we may
  assume, without loss of generality, that for all
  $\theta\in\Theta(\beta^-_n)$, we have
  \[
  \nLim \ntel \insum \log q_\beta(\bar\rho_\beta(\gamma^{-i}\theta)) \ = \
  \int_{\Theta} \log q_\beta(\bar\rho_\beta(\theta)) \ dm(\theta)
  \ =: \ \tilde{\lambda} \ > \ 0
    \]
  and at the same time
  \[
       \nLim \ntel \log q_\beta(\bar\rho_\beta(\gamma^{-n}\theta)) \ = \ 0 \ .
         \]
  Hence, for all $\theta\in\Theta(\beta)$ there exists some
  $n(\theta)>0$ such that for all $n\geq n(\theta)$
  \begin{eqnarray*}
    \lefteqn{\sup\left\{\left.\log\left(F^n_{\beta,\base^{-n}\theta,\alpha}\right)'(r)
          \ \right|\ \alpha\in\kreis,\ r\in K_{\beta^-_l,\delta_k}(\base^{-n}\theta) \right\} } \\
    & \stackrel{~(\ref{e.X})}{\leq} & \sup\left\{\left. -\log r -
        \inergsum\log q_\beta(F^{i}_{\beta,\gamma^{-n}\theta,\alpha}(r))
        \ \right|\ \alpha\in\kreis,\ r\in K_{\beta^-_l,\delta_k}(\base^{-n}\theta) \right\} \\
    & \leq & -\log\bar\rho_\beta(\gamma^{-n}\theta) -
    \insum \log q_\beta(\bar\rho_\beta(\gamma^{-i}\theta)) \ \leq \ -n\tilde\lambda/2 \ .
  \end{eqnarray*}
   This is equivalent to the uniform contraction property provided by
  \eqref{e.random-upperbound}, and from that point on the proof
  proceeds in exactly the same way.
\end{rem}

\subsection{The critical parameters: Proof of Proposition~\ref{p.critical}}

We split the proof into three lemmas, which are the analogues of
statements (a), (b) and (c) of the proposition for the double skew
product. This time, the translation to the original setting is left to
the reader. The first lemma will imply part (b) of the proposition and
also be useful in the proof of part (a).

\begin{lemma}
  \label{l.Prop(b)} Suppose $\beta=\beta_1^m$. Then
  $\psi^+_\beta=0$ $\mu$-a.s.\ for every
  $\mu\in\mathcal{M}_m(g)$ and
  \begin{equation}\label{eq:critical-1}
    \nLim \sup_{\alpha\in\kreis, r\in\R^+} \ntel \inergsum F^n_{\beta,\theta,\alpha}(r)\ = \
    0 \quad\textrm{for } m\textrm{-a.e.}\ \theta\in\Theta  \ .
  \end{equation}
\end{lemma}
\proof
As $\beta=\beta_1^m=e^{-\lambda_m(A)}$, Lemma~\ref{lemmaA} implies
that $\lambda_\mu(\psi^-)=\lambda_\mu(F_\beta,0)\leq 0$ for all
$\mu\in\cM_m(g)$. Therefore, Theorem~\ref{t.convexity} implies that
$\psi^+_\beta=0$ $\mu$-a.s. for each $\mu\in\cM_m(g)$.

Let $\nu\in\cM_m(F_\beta)$ and denote by $\mu\in\cM_m(g)$ its
projection to $\Theta\times\kreis$. If $\nu$ is ergodic, then
$\nu=\mu_\psi$ for some $(g,\mu)$-invariant graph $\psi$ by
Theorem~\ref{t.measures}, and Theorem~\ref{t.convexity} implies now
that $\psi=0$ $\mu$-a.s. It follows that
$\nu(\Theta\times\kreis\times(0,\infty))=0$ for each
$\nu\in\cM_m(F_\beta)$.

Consider the $F_\beta$-forward invariant random compact set
$K(\theta)=\kreis\times[0,1]$ and the additive functions
$\Phi_n(\theta,(\alpha,r))=\sum_{k=0}^{n-1}F^k_{\beta,\theta,\alpha}(r)$.
For each $\nu\in\cM_m(F_\beta)$,
\begin{displaymath}
\overline{\Phi}_\nu
\ = \
\int_{\Theta\times\kreis\times\R^+}\Phi_1 \ d\nu
\ = \
\int_{\Theta\times\kreis\times\{0\}}\Phi_1 \ d\nu
\ = \
0 \ .
\end{displaymath}
Let $\lambda>0$. Then, for $m$-a.e. $\theta\in\Theta$,
\begin{displaymath}
0
\ \leq \
\sup_{\alpha\in\kreis, r\in[0,1]} \Phi_n(\theta,(\alpha,r))
\ \leq \
C_\lambda(\theta)+n\lambda\quad\text{for all $n\in\N$}
\end{displaymath}
by Theorem~\ref{t.random_semiuniform}, and as this is true for each
$\lambda>0$, the claim (\ref{eq:critical-1}) restricted to $r\in[0,1]$
follows at once.  As the fibre maps are monotone and bounded by $1$,
the extension to $r\in\R^+$ is immediate.  \qed\medskip

\begin{lemma}
  \label{l.Prop(a)} Suppose $\beta=\beta^m_1<\beta_2^m$. Then for
  $m$-a.e.\ $\theta\in\Theta$ we have $\psi^+_\beta(\theta,\alpha)=0$
  for all $\alpha\in\kreis$ and there is a set $J(\theta) \ssq\N$ of
  asymptotic density $0$ such that
  \begin{equation}
    \label{eq:density-one-1}
    \lim_{\substack{n\to\infty \\ n\notin J(\theta)}} F^n_{\beta,\theta,\alpha}(r) \ = \ 0 \quad
      \textrm{ for all } (\alpha,r)\in\kreis \times[0,\infty) \ .
  \end{equation}
\end{lemma}
\proof As $\psi_\beta^+$ is increasing in $\beta$, the fact that
$\psi^+_\beta(\theta,\alpha)=0$ when $\alpha\neq\phi_u(\theta)$
follows from Lemma~\ref{l.intermediate_upperbounding}. For
$\alpha=\phi_u(\theta)$, this follows from the fact that
$\lambda_{m_{\phi_u}}(F_\beta,0)=0$.

By Lemma~\ref{lemmaA} we have
$\lambda_{m_{\phi^s}}(F_\beta,0)=-\lambda_m(A)+\log\beta_1(m)=-2\lambda_m(A)<0$.
Consequently, for $m$-a.e.\ $\theta\in\Theta$
\[
\nLim F^n_{\beta,\theta,\phi_s(\theta)}(1) \ = \ 0 \ .
\]
Further, Lemma~\ref{l.Prop(b)} implies that for $m$-a.e.
$\theta\in\Theta$
\begin{equation}
\nLim\frac{1}{n}\sum_{i=0}^{n-1}F^i_{\beta,\theta,\alpha}(1)=0
\quad\text{for all $(\alpha,r)\in\kreis\times[0,\infty)$.}
\end{equation}
Hence, for $m$-a.e. $\theta$ there is a set $J(\theta)\subset\N$
 of asymptotic density $0$ such that
\begin{equation}\label{eq:density-one-2}
\lim_{\substack{n\to\infty \\ n\notin J(\theta)}}F^n_{\beta,\theta,\phi^u(\theta)}(1)=0.
\end{equation}
In order to prove that, along the same subsequence $\N\smin J(\theta)$
of asymptotic density $1$, $F^n_{\beta,\theta,\alpha}(1)\to0$ for all
$\alpha\in\kreis$, we use a modification of the proof of
Lemma~\ref{lemma:replaceL5.8}.  Choose
$\alpha,\alpha'\in\kreis\smin\{\phi^s(\theta)\}$, $r,r'>0$, and let
$(\theta_k,\alpha_k,r_k)=F^k_{\beta,\theta,\alpha}(r)$ and define
$(\theta_k,\alpha_k',r_k')$ in the analogous way. Let
\begin{equation}
\ell=\ell(n)=
\begin{cases}
\max\left\{k\leq n:
r_k'\geq r_k\right\}&\text{if such a $k$ exists}\\
0&\text{otherwise.}
\end{cases}
\end{equation}
If $\ell=n$, then $\frac{r_n'}{r_n}\geq 1$. For $\ell<n$,
Lemma~\ref{lemma:main-comparison} implies
\begin{equation}
  \frac{r_n'}{r_n}
  \geq
  \min\left\{
    \frac{r_\ell'}{r_\ell},1\right\}\cdot
  \frac{\Omega_{n-\ell}(\theta_\ell,\alpha_\ell')}{\Omega_{n-\ell}(\theta_\ell,\alpha_\ell)} \ .
\end{equation}
If $\ell>0$, then $r_\ell'\geq r_\ell$. If $\ell=0$, then
$r_\ell=r$ and $r_\ell'=r'$.
In any case,
\begin{equation}
  \frac{r_n'}{r_n}
  \geq
  \min\left\{
    \frac{r'}{r},1\right\}\cdot
  \frac{\Omega_{n-\ell}(\theta_\ell,\alpha_\ell')}{\Omega_{n-\ell}(\theta_\ell,\alpha_\ell)} \ .
\end{equation}

Now, if $(\cos2\pi\alpha,\sin2\pi\alpha)^t =
av^u(\theta)+bv^s(\theta)$, where $v^u(\theta),v^s(\theta)$ are the
unit vectors contained in the invariant subspaces
$E^u(\theta),E^s(\theta)$ of the Oseledets splitting (see
Theorem~\ref{oseledets}), then $\kLim
\frac{\Omega_k(\theta,\alpha)}{\Omega_k(\theta,\phi^u(\theta))} = a$.  Likewise,
if $(\cos2\pi\alpha',\sin2\pi\alpha')^t =
a'v^u(\theta)+b'v^s(\theta)$ then we have that $\kLim
\frac{\Omega_k(\theta,\alpha')}{\Omega_k(\theta,\phi^u(\theta))} = a'$.
Together, this yields
\[
\kLim \frac{\Omega_k(\theta,\alpha')}{\Omega_k(\theta,\alpha)} \ = \
\frac{a'}{a} \ .
\]
Consequently, since
\[
\frac{\Omega_{n-\ell}(\theta_\ell,\alpha'_\ell)}{\Omega_{n-\ell}(\theta_\ell,\alpha_\ell)}
\ = \
\frac{\Omega_{n}(\theta,\alpha')}{\Omega_{n}(\theta,\alpha)} \cdot
\frac{\Omega_{\ell}(\theta,\alpha)}{\Omega_{\ell}(\theta,\alpha')}
\]
and the two factors on the right converge to $a/a'$ and $a'/a$ as $n$
and $\ell$ go to infinity, there exists a constant
$C(\theta,\alpha,\alpha')>$ such that
\begin{equation}
  \frac{F_{\beta,\theta,\alpha'}^n(r')}{F_{\beta,\theta,\alpha}^n(r)}
  =
  \frac{r_n'}{r_n}
  \geq
  \min\left\{\frac{r'}{r},1\right\}\cdot C(\theta,\alpha,\alpha')
  \quad\text{for all }n.
\end{equation}
For $\alpha'=\phi^u(\theta)$ and $r'=1$, this is the estimate needed to
infer (\ref{eq:density-one-1}) from (\ref{eq:density-one-2}).
\qed\medskip

\begin{lemma}
  \label{l.Prop(c)} Suppose $\beta=\beta_2^m>\beta_1^m$. Then
\[
\psi^+_\beta(\theta,\alpha) \ = \ 0 \quad \textrm{for }
m\textrm{-a.e.\ }\theta\in\Theta \textrm{ and all }
\alpha\in\kreis\smin\{\phi_u(\theta)\} \ .
\]
\end{lemma}
\proof As $\log\beta=\lambda_m(A)$, Lemma~\ref{lemmaA} implies that
$\lambda_{m_{\phi_s}}(F_\beta,0)=0$. Hence
$\psi_\beta^+(\theta,\alpha)=0$ for \mbox{$m_{\phi_s}$-a.e.}
$(\theta,\alpha)$ by Theorem~\ref{t.convexity}, which means that
$\psi_\beta^+(\theta,\alpha)=0$ for $m$-a.e. $\theta\in\Theta$ and
$\alpha=\phi_s(\theta)$.

For $\alpha,\alpha'\in\kreis\smin\{\phi_u(\theta)\}$ and $n\in\N$ we
let $\bar\theta=\gamma^{-n}\theta$, $\bar\alpha=g_\theta^{-n}(\alpha)$
and $\bar\alpha'=g_\theta^{-n}(\alpha')$. Further, let $\bar
r_k=F^k_{\beta,g^{-k}(\theta,\alpha)}(1)$ and $\bar
r_k'=F^k_{\beta,g^{-n}(\theta,\alpha')}(1)$ and define
\begin{equation}
\ell=\ell(n)=
\begin{cases}
\max\left\{k\leq n:
\bar r_k'\geq \bar r_k\right\}&\text{if such a $k$ exists}\\
0&\text{otherwise.}
\end{cases}
\end{equation}
Then Lemma~\ref{lemma:main-comparison} applied to
$\bar\theta,\bar\alpha,\bar\alpha'$ and $\bar r=\bar r'=1$ yields
\begin{equation}
  \label{e.critical(c)1}
  \frac{F^n_{\beta,g^{-n}(\theta,\alpha')}(1)}{F^n_{\beta,g^{-n}(\theta,\alpha)}(1)}
  \ = \ \frac{\bar r_n'}{\bar r_n} \ \geq \ \
  \frac{\Omega_{n-\ell}(\bar\theta_\ell,\bar\alpha_\ell')}{\Omega_{n-\ell}(\bar\theta_\ell
    ,\bar\alpha_\ell)} \ .
\end{equation}
Suppose $(\cos2\pi\alpha,\sin2\pi\alpha)=av^u(\theta)+bv^s(\theta)$ and
$(\cos2\pi\alpha',\sin2\pi\alpha')=a'v^u(\theta)+b'v^s(\theta)$, where
$v^u(\theta),v^s(\theta)$ are defined as in the proof of the previous
lemma. We have
\[
\frac{\Omega_{n-\ell}(\bar\theta_\ell,\bar\alpha_\ell')}{\Omega_{n-\ell}(\bar\theta_\ell,
  \bar\alpha_\ell)} \ = \
\frac{\Omega_{-n}(\bar\theta,\bar\alpha)}{\Omega_{-n}(\bar\theta,\bar\alpha')}\cdot
\frac{\Omega_{-\ell}(\bar\theta,\bar\alpha')}{\Omega_{-\ell}(\bar\theta,\bar\alpha)}
\ .
\]
Similar as in the previous proof, the two factors on the right
converge to $b/b'$ and $b'/b$, respectively, as $n$ and $\ell$ go to
infinity. For this reason, there exists a constant $\bar
C(\theta,\alpha,\alpha')$ such that $\frac{\bar r_n'}{\bar r_n} \geq
\bar C(\theta,\alpha,\alpha')$. Using (\ref{e.critical(c)1}), this
means that
\[
F^n_{\beta,g^{-n}(\theta,\alpha')}(1) \ \geq \ \bar
C(\theta,\alpha,\alpha') \cdot F^n_{\beta,g^{-n}(\theta,\alpha)}(1) \ .
\]
If we let $\alpha'=\phi_s(\theta)$, then in the limit $n\to\infty$ we
obtain
\[
0\ = \ \psi^+_\beta(\theta,\phi_s(\theta)) \ \geq \ \bar
C(\theta,\alpha,\alpha')\cdot \psi^+_\beta(\theta,\alpha)
\]
and thus $\psi^+_\beta(\theta,\alpha)=0$ for $m$-a.e.
  $\theta\in\Theta$ and all $\alpha\in\kreis\smin\{\phi_u(\theta)\}$
 as required.  \qed\medskip

\section{Continuous-time models}
\label{CtsModels}

The classical Hopf bifurcation takes place in continuous-time dynamical systems
generated by planar vector fields. Its discrete-time analogue, whose
nonautonomous version we have considered so far, is often called Neimark-Sacker
bifurcation. However, since a Hopf bifurcation for a planar flow corresponds to
a Neimark-Sacker bifurcation of the corresponding time-one map, this is a minor
distinction.  Nevertheless, as we have made quite specific additional
assumptions on the considered models, it is important to note that
continuous-time systems with similar properties exist. The aim of this section
is to provide examples of continuous-time flows, generated by non-autonomous
planar vector fields, whose time-one maps have a similar structure as the maps
considered in the previous sections and therefore exhibit the same bifurcation
pattern. We will only sketch the details and concentrate on the deterministic
setting. Randomly forced examples can be produced in an analogous way.

Suppose that $\Theta$ is a compact metric space and $\omega :
\R\times\Theta \to \Theta,\ (t,\theta)\mapsto \omega_t\theta$ is a
continuous flow on $\Theta$. First, consider the linear
two-dimensional ordinary differential equation
\begin{equation}
  \label{eq:1}
  \twovector{x}{y}' \ = \ B(\omega_t \theta) \twovector{x}{y} \
\end{equation}
with continuous $B=\twomatrix{\hat a_\theta}{\hat b_\theta}{\hat
c_\theta}{\hat
  d_\theta} :\Theta\to \textrm{sl}(2,\R)$. The time-one map of the generated
flow is given by a linear $\sltr$-cocycle $(\omega_1,A_1)$, with
continuous $A_1:\Theta\to\sltr$ obtained by integrating $B$ along
the orbits of $\omega$. In polar coordinates
$(\alpha,r)=\left(\arctan(y/x)/\pi \bmod 1, \sqrt{x^2+y^2}\right)$,
equation (\ref{eq:1}) is written as
\begin{eqnarray}
  \label{eq:2}
  \alpha' & = &  \frac{1}{\pi} \left(\hat c_{\omega_t\theta}\cos^2\pi\alpha +
    (\hat d_{\omega_t\theta}-\hat a_{\omega_t\theta})\cos\pi\alpha\sin\pi\alpha -
   \hat b_{\omega_t\theta}\sin^2\pi\alpha \right) \\
  r' & = & \gamma(\omega_t\theta,\alpha) \cdot r \ \label{eq:2b}
\end{eqnarray}
with $\gamma(\theta,\alpha)= \hat a_\theta \cos^2\pi\alpha + (\hat
b_\theta+\hat c_\theta)\sin\pi\alpha\cos\pi\alpha + \hat
d_\theta\sin^2\pi\alpha$. The time-one map of this system is given by
\begin{equation}
  \label{eq:3}
  \widehat F(\theta,\alpha,r) \ = \ \left(g_1(\theta,\alpha),
\Omega_1(\theta,\alpha)\cdot r\right) \ ,
\end{equation}
where $g_1(\theta,\alpha) = (\omega_1\theta,g_{1,\theta}(\alpha))$ is
the projective action of the cocycle $(\omega_1,A_1)$ and the factor
$\Omega_1(\theta,\alpha)$ is obtained by integrating (\ref{eq:2b}).

In order to introduce a bifurcation parameter and to make the fibre
maps concave in $r$, we now replace (\ref{eq:1})
by
\begin{equation}
  \twovector{x}{y}' \ = \ \left(\strut B(\omega_t \theta)+(\beta+\eta(r))E\right)
\twovector{x}{y} \
\end{equation}
with a non-positive decreasing $\mathcal{C}^2$-function $\eta:\R^+\to
\R$ such that $r\mapsto \eta(r)\cdot r$ is concave, $\eta(0)=0$ and
$\lim_{r\to\infty}\eta(r)=-\infty$.  This results in the modified
equation
\begin{equation}
  \label{eq:4}
  r' \ = \ \left(\gamma(\omega_t\theta,\alpha)+\beta+\eta(r)\right)\cdot r \
\end{equation}
replacing (\ref{eq:2b}), while (\ref{eq:2}) is unaffected.
The resulting time-one map will be
of the form
\begin{equation}
  \label{eq:5}
  F_\beta(\theta,\alpha,r) \ = \ (g_1(\theta,\alpha),F_{\beta,\theta,\alpha}(r)) \ ,
\end{equation}
with fibre maps $F_{\beta,\theta,\alpha}:\R^+\to\R^+$ that have the
following properties:
\begin{itemize}
\item $F_{\beta,\theta,\alpha}$ is a strictly increasing
  $\mathcal{C}^2$-function;
\item $F_{\beta,\theta,\alpha}(0)=0$;
\item $F_{\beta,\theta,\alpha}'(0)=e^\beta\cdot \Omega_1(\theta,\alpha)$;
\item $\beta\mapsto F_{\beta,\theta,\alpha}(r)$ is strictly increasing
  for all $\theta\in\Theta,\ \alpha\in\kreis$ and $r>0$.
\item $F_{\beta,\theta,\alpha}$ is strictly concave (due to the
  concavity of the right side of (\ref{eq:4})).
\item $\sup_{\theta,\alpha,r} F_{\beta,\theta,\alpha} < \infty$ (due
  to the facts that $\lim_{r\to\infty} \eta(r)=-\infty$ and $\gamma$
  is uniformly bounded).
\end{itemize}
While the fibre maps of $F_\beta$ are not exactly of the same form as
in Lemma~\ref{l.conjugacy}(iii), they have all the qualitative
features that were used in our analysis. The specific form of the maps
in (\ref{f_beta}) was chosen for reasons of presentation and
readability, but all arguments go through in the generality needed to
treat maps with the above properties.  Thus, the family
$(F_\beta)_{\beta\in\R}$ satisfies all the assertions of
Theorem~\ref{theorem_topological} with critical parameters
$\beta_1=-\lmax(A_1)$ and $\beta_2=\lmax(A_1)$.  Analogous statements
with obvious modifications for the continuous-time case hold for the
flow generated by (\ref{eq:2}) and (\ref{eq:4}).

\section{Simulations} \label{Examples}
In this section, we illustrate the preceding results by an explicit example
$f_\beta:\kreis\times\R^2\rightarrow\kreis\times\R^2$ with skew product
structure $(\theta,v)\mapsto(\gamma \theta, f_{\beta,\theta}(v))$. For
simplicity, the base transformation is chosen to be an irrational rotation of
the circle, that is, $\gamma:\kreis\rightarrow\kreis$,
$\theta\mapsto\theta+\varrho \mod 1$, where $\varrho$ is the golden mean. The
fibre maps are defined by
\begin{equation} f_{\beta,\theta}(v) \ = \
  \left\{\begin{array}{cl} h(\beta\|v\|) A(\theta) \frac{v}{\|v\|} &
      \textrm{if } v\neq 0\\ \ \\0 & \textrm{if } v=0
\end{array}\right. \quad ,
\end{equation}
where $h(x)=\frac{1}{3\sqrt{c}}\arctan(x)$ and $A(\theta)=\left(
                                          \begin{array}{cc}
                                            c^{-1/2} & 0 \\
                                            0 & c^{1/2} \\
                                          \end{array}
                                        \right)
                                        \left(
                                          \begin{array}{cc}
                                            \cos(2\pi\theta) &
                                            \sin(2\pi\theta)\\
                                            -\sin(2\pi\theta) & \cos(2\pi\theta)
                                            \\
                                          \end{array}
                                        \right)=:
                                        \left(
                                          \begin{array}{cc}
                                            a_\theta & b_\theta \\
                                            c_\theta & d_\theta \\
                                          \end{array}
                                        \right).
$

The map $f_\beta$ is easily seen to satisfy (D1)--(D3) as well as
condition (\ref{e.scaling}). Therefore, $f_\beta$ satisfies Theorem
\ref{theorem_topological} and for $c=1/2$, the bifurcation
parameters (determined by the maximal exponential expansion rate of
the cocycle $(\gamma,A)$, see \cite[Section 4.1]{herman:1983}) are
given by $\beta_1=3\sqrt{2}e^{-\lambda(A)}=3\sqrt{2}\left(
\frac{2\sqrt{2}}{3}\right)=4$ and
$\beta_2=3\sqrt{2}e^{\lambda(A)}=3\sqrt{2}\left(
\frac{1}{\sqrt{2}}+\frac{1}{2\sqrt{2}}\right)=4.5$.

\begin{figure}[h!]
\subfigure[$\beta_1<\beta=4.080<\beta_2$]{\epsfig{file=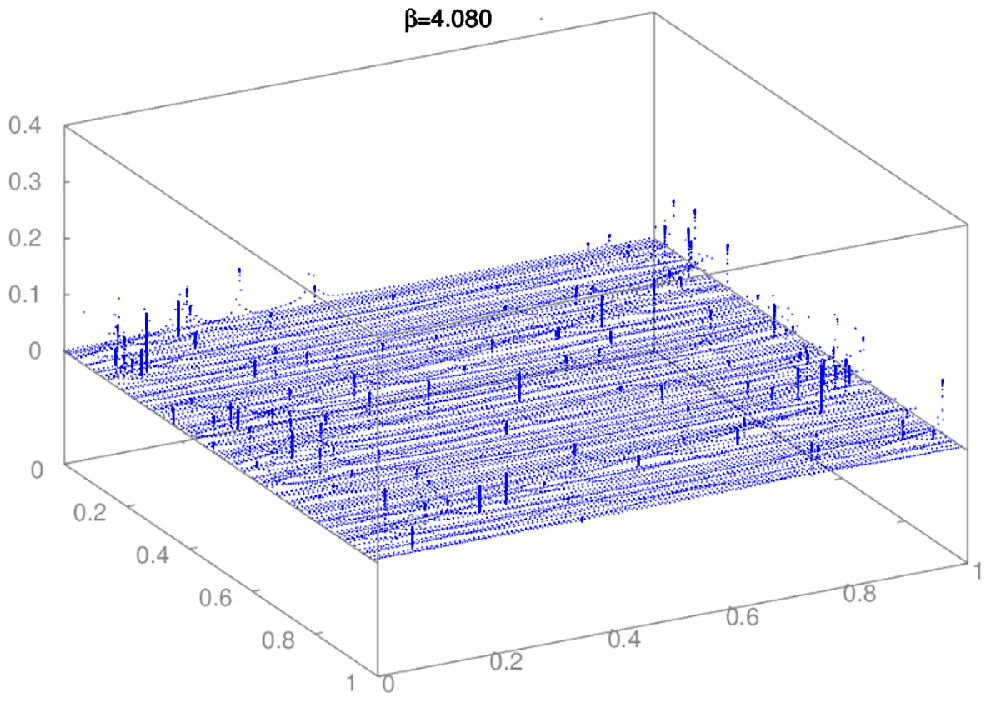,
width=0.5\linewidth}}
  \subfigure[$\beta_1<\beta=4.475<\beta_2$]{\epsfig{file=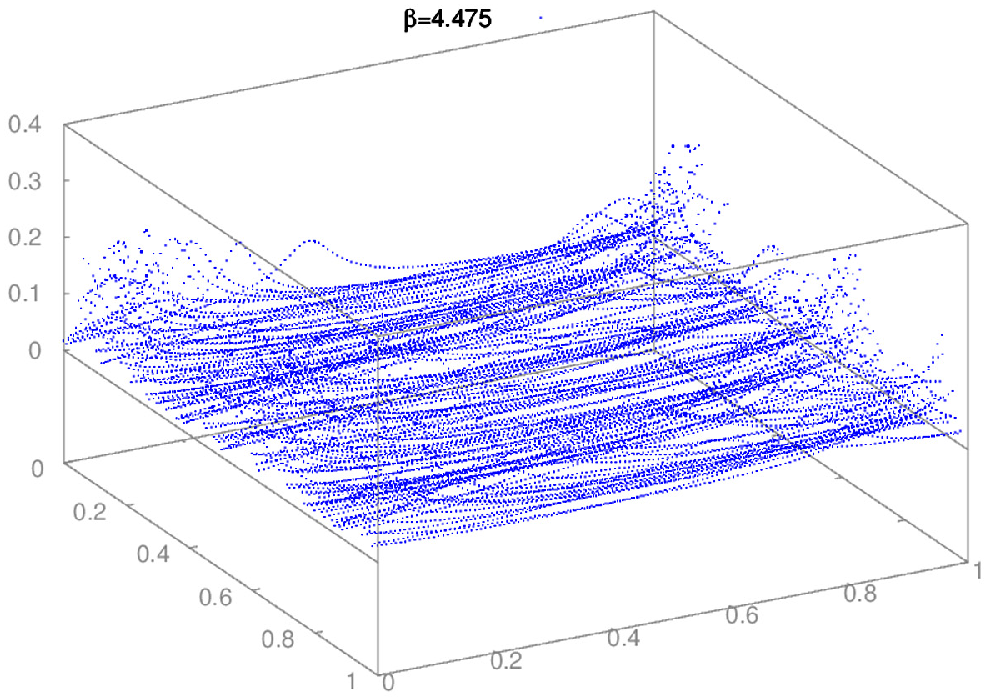, width=0.5\linewidth}}
\subfigure[$\beta=4.650>\beta_2$]{\epsfig{file=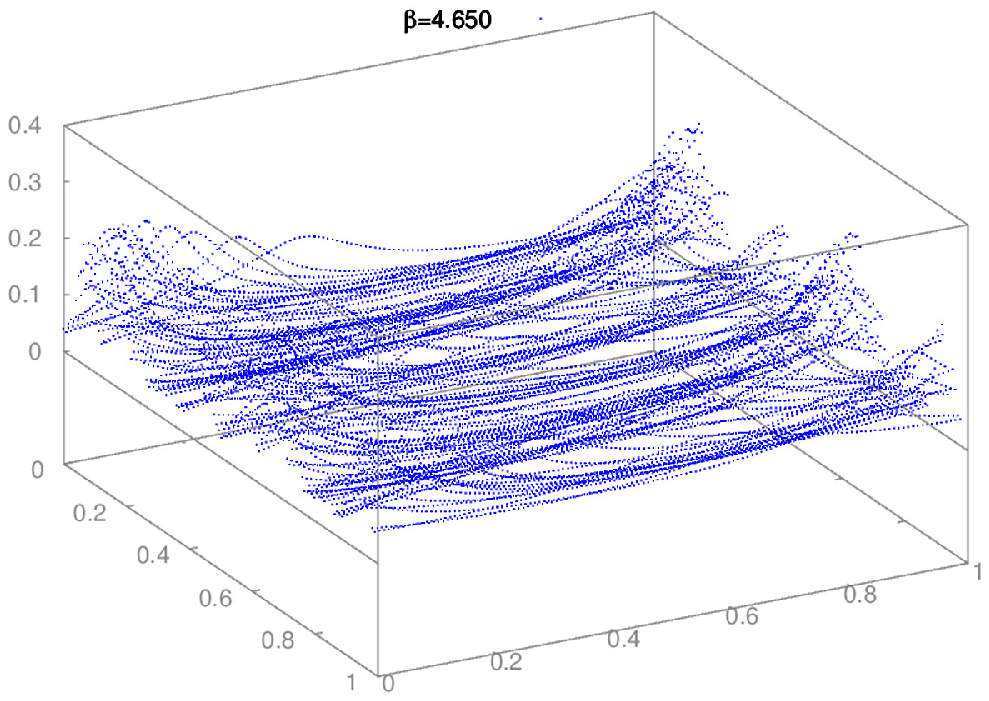,
width=0.5\linewidth}}
  \subfigure[$\beta=6.080>\beta_2$]{\epsfig{file=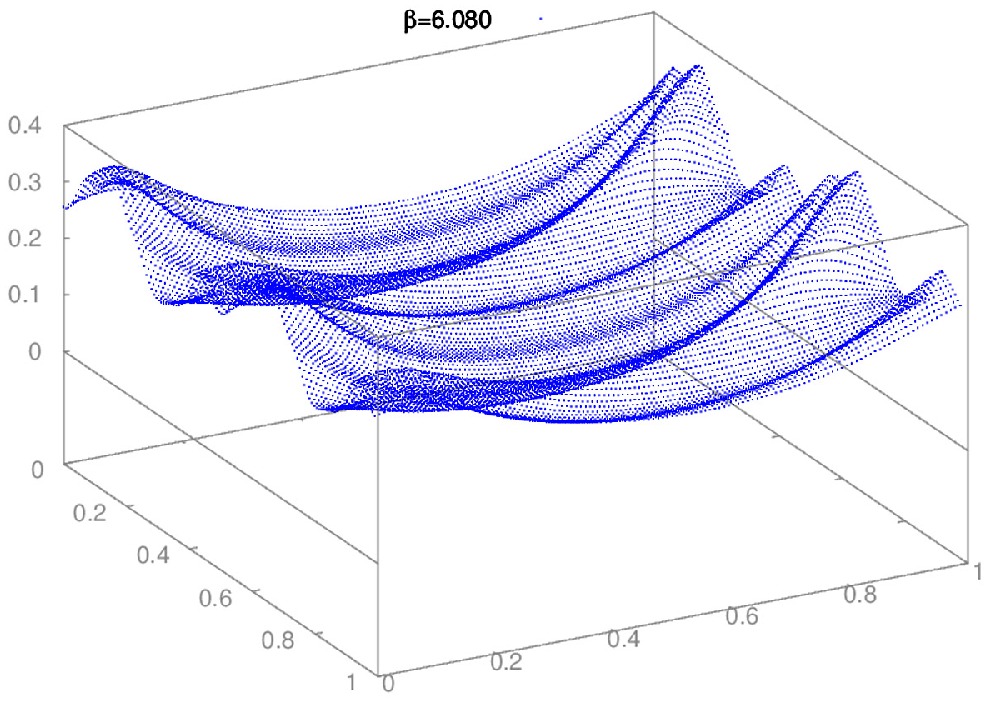, width=0.5\linewidth}}
\caption{The global attractor of the induced polar coordinate system $F_\beta$.}\label{figure_polar}
 \end{figure}
\newpage

As in Section \ref{top.doubleskew}, we use polar coordinates in
order to study the induced system
$F_\beta:\kreis\times\kreis\times[0,\infty)\rightarrow\kreis\times\kreis\times[0,\infty)$,
given by $F_\beta(\theta,\alpha,r)=(\gamma\theta,
g_\theta(\alpha),F_{\beta,\theta,\alpha}(r))$, where
$g_\theta(\alpha)=\frac{1}{\pi}\arctan\left(
\frac{c_\beta+d_\theta\tan\pi\alpha}{a_\beta+b_\theta\tan\pi\alpha}\right)
\mod 1$, and $F_{\beta,\theta,\alpha}(r)=\frac{\arctan(\beta
r)}{3\sqrt{2}}|| A(\theta)(\cos\pi\alpha,\sin\pi\alpha)||$

Figures \ref{figure_polar}(a)--\ref{figure_polar}(d) illustrate the
behaviour of the induced polar coordinate system $F_\beta$. Figures
\ref{figure_polar}(a) and \ref{figure_polar}(b), show the global
attractor shortly after the first critical parameter $(\beta_1)$ and
just before the second $(\beta_2)$, respectively. Figure
\ref{figure_polar}(c) shows $\mathcal{A}_\beta(\theta)$ shortly after
$\beta_2$ where the invariant torus has just formed, and finally,
Figure \ref{figure_polar}(d) shows the split off torus far from the
bifurcation (all in polar coordinates). For the same values of
$\beta$, Figures \ref{figure_original}(a)--\ref{figure_original}(d)
illustrate the behaviour of the global attractor for the original
system $f_\beta$.

\begin{figure}[h!]
\subfigure[$\beta_1<\beta=4.080<\beta_2$]{\epsfig{file=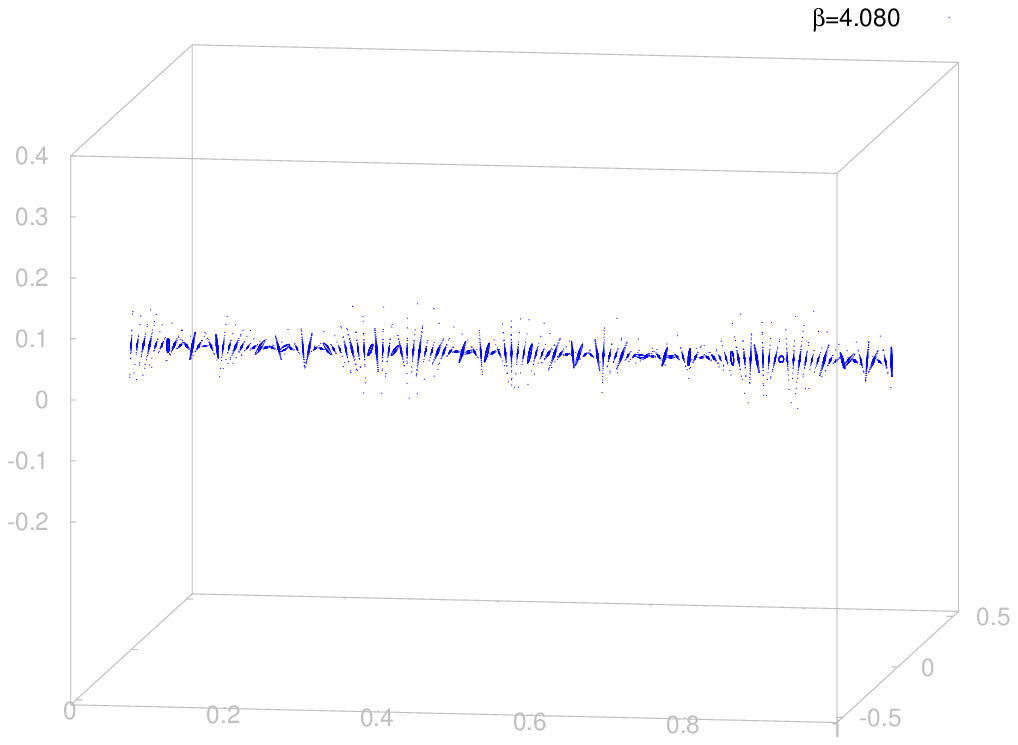,
width=0.49\linewidth}}
  \subfigure[$\beta_1<\beta=4.475<\beta_2$]{\epsfig{file=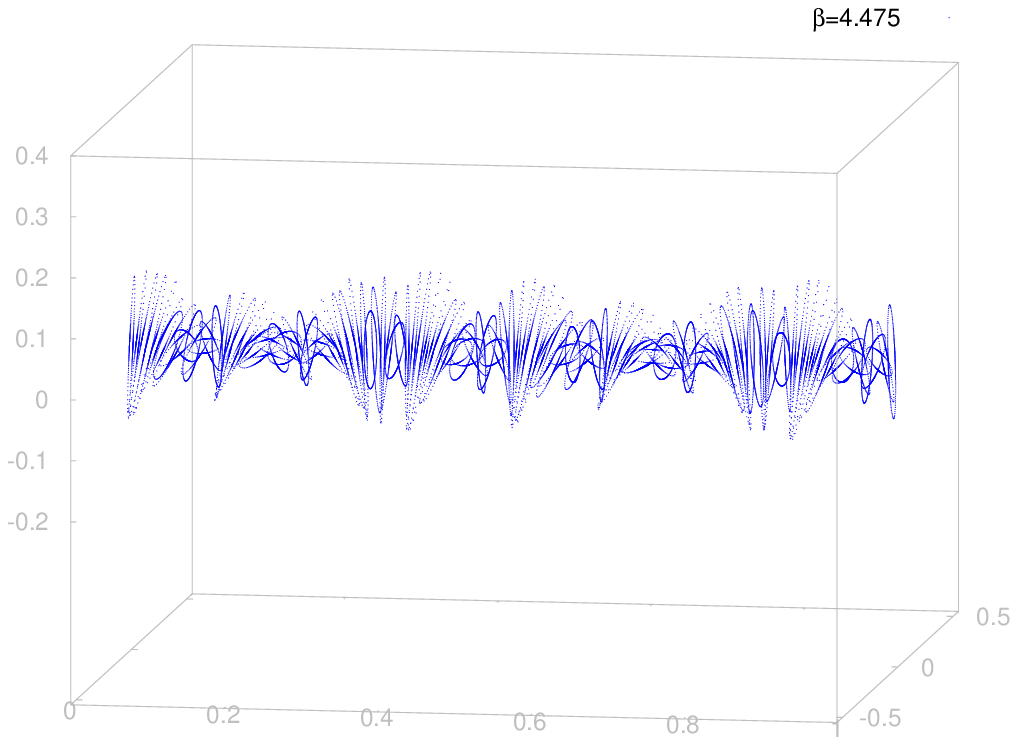, width=0.49\linewidth}}
\subfigure[$\beta=4.650>\beta_2$]
{\epsfig{file=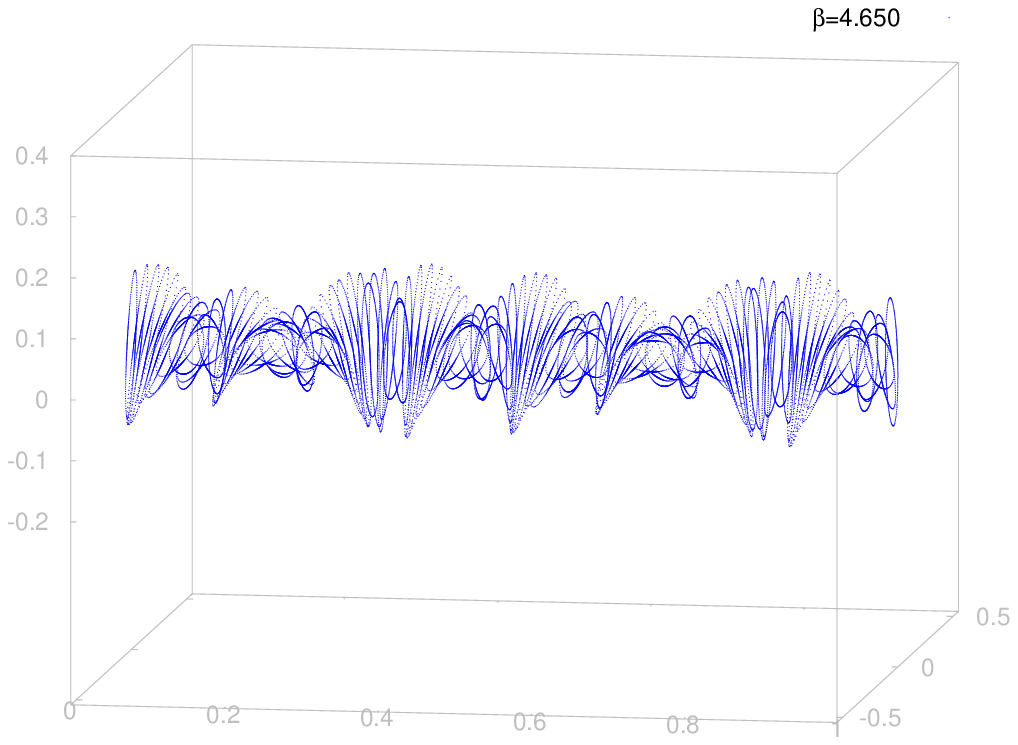, width=0.49\linewidth}}
  \subfigure[$\beta=6.080>\beta_2$]{\epsfig{file=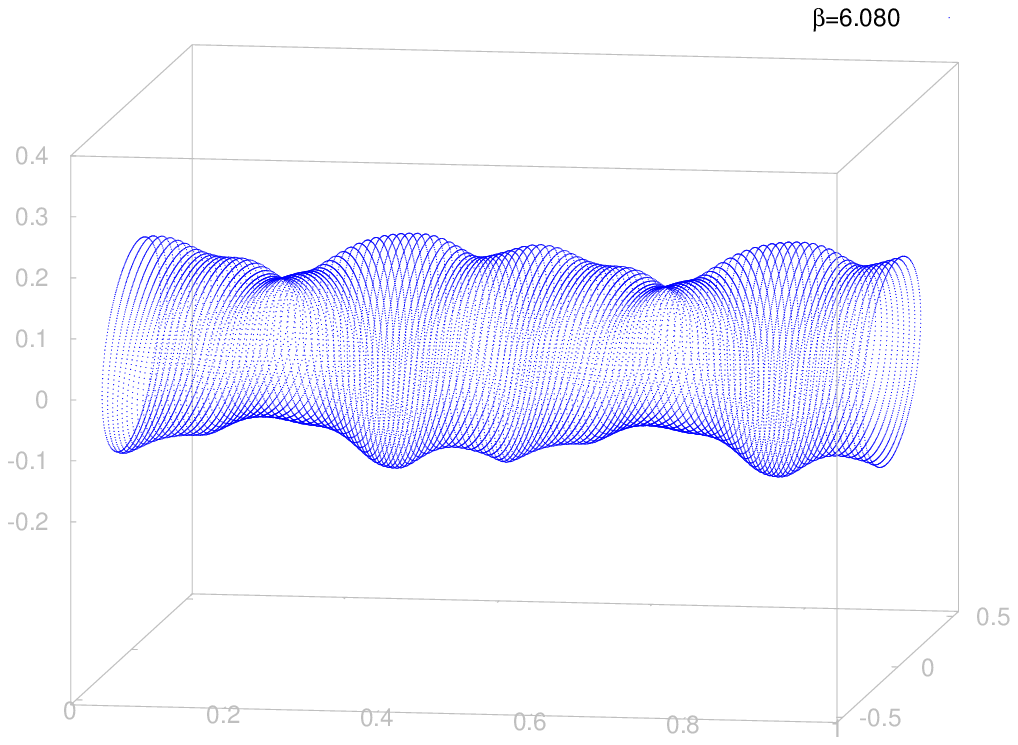, width=0.49\linewidth}}
\caption{The global attractor of the original system $f_\beta$.}\label{figure_original}
 \end{figure}

\begin{figure}[h!]
  \subfigure[$\beta_1<\beta=4.475<\beta_2$]{
     \epsfig{file=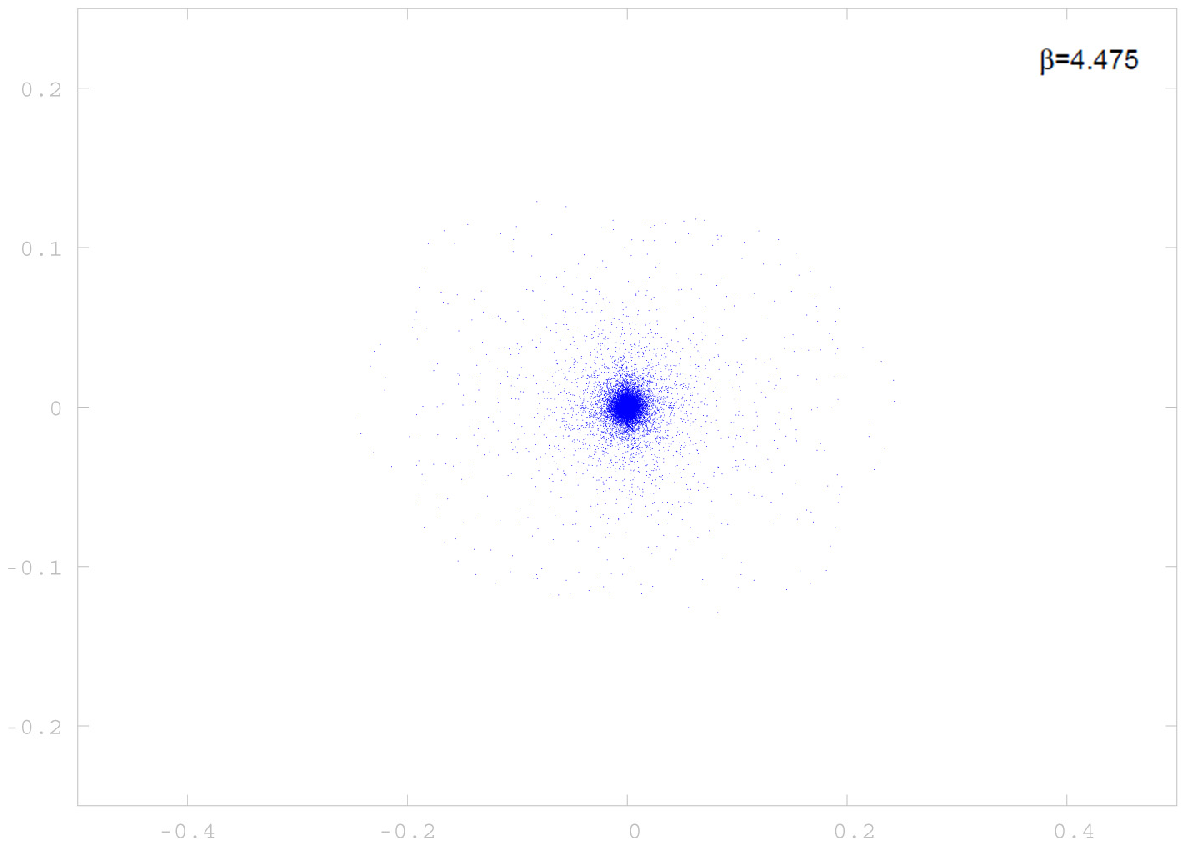,
      width=0.49\linewidth}}
  \subfigure[$\beta=4.465>\beta_2$]{
    \epsfig{file=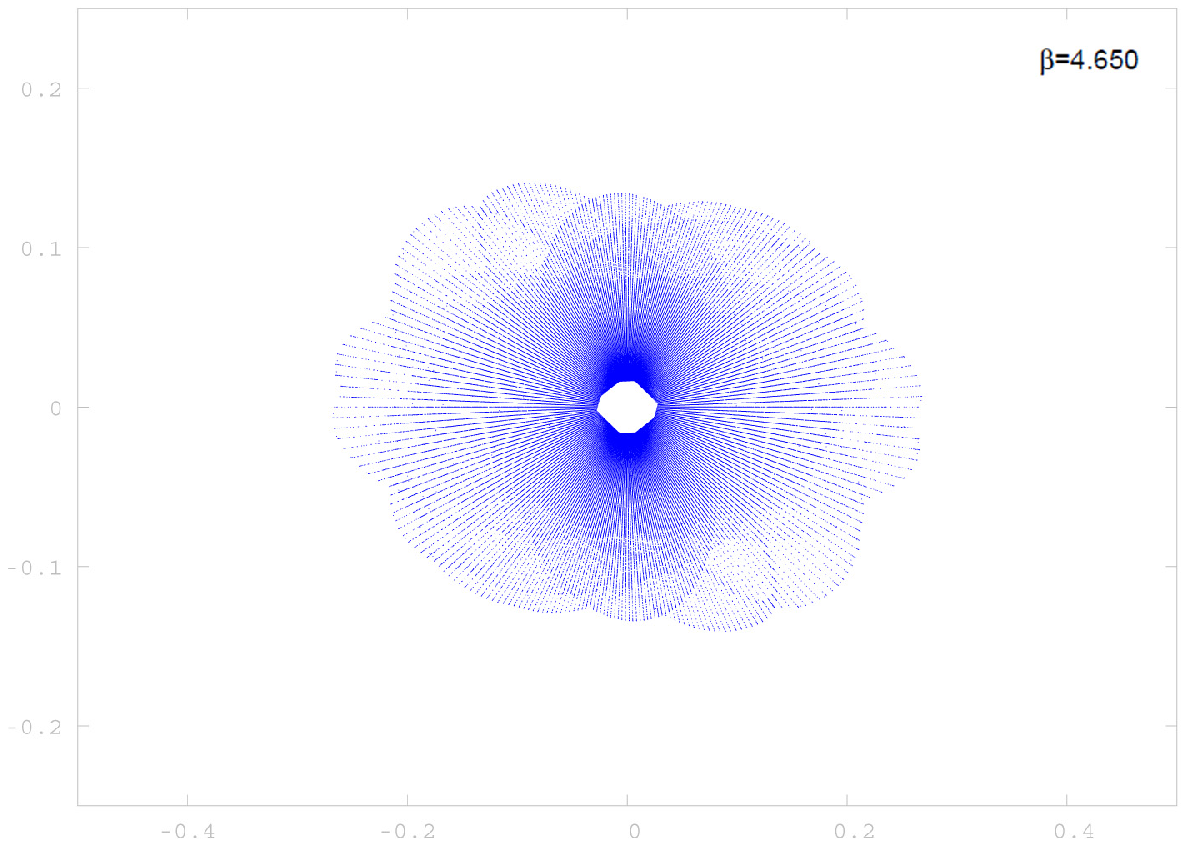,
      width=0.49\linewidth}}
  \caption{2D projection of the global attractor onto $\R^2$}\label{figure_2d}
 \end{figure}

Finally, Figures 7.3(a) and 7.3.(b) show a 2D projection of the
torus onto $\R^2$ for $\beta=4.475$ (before the torus has formed),
and $\beta=4.465$ (after the torus has split off), respectively.

 All pictures were produced by using a mixture of pullback and forward iteration
 for a fixed grid of $(\theta,\alpha)$-coordinates.


\end{document}